\documentclass{amsart}
\usepackage[utf8]{inputenc}
\usepackage[all]{xy}
\usepackage{url}
\usepackage{hyperref}
\usepackage{amsfonts}
\usepackage{amsmath}
\usepackage{amssymb}
\usepackage{extpfeil}
\usepackage{amsthm}
\usepackage{mathtools}
\usepackage{tikz-cd}
\usepackage{xcolor}
\usepackage{caption}
\newcommand{\op}{\textup{op}}

\newcommand{\id}{\textup{id}}

\newcommand{\C}{\mathcal{C}}

\newcommand{\D}{\mathcal{D}}

\newcommand{\Fun}{\textup{Fun}}
\usepackage{stackrel}

\numberwithin{equation}{section}
\theoremstyle{definition}
\newtheorem{defn}[equation]{Definition}

\newtheorem{rmk}[equation]{Remark}

\newtheorem{cor}[equation]{Corollary}
\theoremstyle{plain}
\newtheorem{thm}[equation]{Theorem}
\newtheorem{lem}[equation]{Lemma}
\newtheorem{prop}[equation]{Proposition}

\newcommand{\var}{\textup{Var}}

\newcommand{\smvar}{\textup{SmVar}}
\newcommand{\compvar}{\textup{Comp}}
\newcommand{\comp}{\textup{Comp}}
\newcommand{\smcomp}{\textup{SmComp}}
\newcommand{\smcompvar}{\textup{SmComp}}

\newcommand{\Span}{\textup{Span}}
\renewcommand{\span}{\Span}
\newcommand{\shv}{\textup{Sh}}
\newcommand{\sh}{\textup{Sh}}
\newcommand{\map}{\textup{Map}}
\renewcommand{\hom}{\textup{Hom}}

\newcommand{\fib}{\textup{fib}}
\newcommand{\set}{\textup{Set}}

\newcommand{\colim}{\textup{colim}}
\renewcommand{\bar}{\overline}

\newcommand{\X}{\mathcal{X}}
\newcommand{\Y}{\mathcal{Y}}
\newcommand{\psh}{\textup{PSh}}
\newcommand{\hsh}{\textup{HSh}}
\newcommand{\hshv}{\textup{HSh}}
\renewcommand{\S}{\mathcal{S}}
\newcommand{\sset}{\textup{sSet}}
\newcommand{\exc}{\textup{exc}}

\makeatletter
\newcommand*{\da@rightarrow}{\mathchar"0\hexnumber@\symAMSa 4B }
\newcommand*{\da@leftarrow}{\mathchar"0\hexnumber@\symAMSa 4C }
\newcommand*{\xdashrightarrow}[2][]{
  \mathrel{
    \mathpalette{\da@xarrow{#1}{#2}{}\da@rightarrow{\,}{}}{}
  }
}
\newcommand{\xdashleftarrow}[2][]{
  \mathrel{
    \mathpalette{\da@xarrow{#1}{#2}\da@leftarrow{}{}{\,}}{}
  }
}
\newcommand*{\da@xarrow}[7]{

  \sbox0{$\ifx#7\scriptstyle\scriptscriptstyle\else\scriptstyle\fi#5#1#6\m@th$}
  \sbox2{$\ifx#7\scriptstyle\scriptscriptstyle\else\scriptstyle\fi#5#2#6\m@th$}
  \sbox4{$#7\dabar@\m@th$}
  \dimen@=\wd0 
  \ifdim\wd2 >\dimen@
    \dimen@=\wd2 
  \fi
  \count@=2 
  \def\da@bars{\dabar@\dabar@}
  \@whiledim\count@\wd4<\dimen@\do{
    \advance\count@\@ne
    \expandafter\def\expandafter\da@bars\expandafter{
      \da@bars
      \dabar@ 
    }
  }
  \mathrel{#3}
  \mathrel{
    \mathop{\da@bars}\limits
    \ifx\\#1\\
    \else
      _{\copy0}
    \fi
    \ifx\\#2\\
    \else
      ^{\copy2}
    \fi
  }
  \mathrel{#4}
}
\makeatother

\title[A descent principle for compactly supported extensions]{A descent principle for compactly supported extensions of functors}

\author{Josefien Kuijper}
\email{josefien.kuijper@math.su.se}
\address{Matematiska instutionen\\Stockholms universitet\\106 91 Stockholm\\Sweden }

\thanks{The author is supported by ERC-2017-STG 759082}

\begin{document}

\begin{abstract}
A characteristic property of cohomology with compact support is the long exact sequence that connects the compactly supported cohomology groups of a space, an open subspace and its complement. Given an  arbitrary cohomology theory of algebraic varieties, one can ask whether a compactly supported version exists, satisfying such a long exact sequence. This is the case whenever the cohomology theory satisfies descent for abstract blowups (also known as proper cdh descent). We make this precise by proving an equivalence between certain categories of hypersheaves. We show how several classical and non-trivial results, such as the existence of a unique weight filtration on cohomology with compact support, can be derived from this theorem.
\end{abstract} 

\maketitle

\section{Introduction}
A characteristic property of cohomology with compact support is the existence of a long exact sequence
\begin{equation}\label{eq:les_compact_support}
    \cdots \to H^n_c(U) \to H^n_c(X) \to H^n_c(X\setminus U) \to H^{n+1}_c(U) \to \cdots  
\end{equation}
for $U\subseteq X$ an open subspace. Given an arbitrary cohomology theory of, say, algebraic varieties, one can wonder if it makes sense to \textit{define} a compactly supported cohomology theory such that a long exact sequence like (\ref{eq:les_compact_support}) exists by construction. Let $\var$ denote the category of algebraic varieties over a field $k$, and let $\C$ be a triangulated category with a t-structure. We assume that a cohomology theory arises from a functor 
$$F:\var^\op \to \C$$
as the homotopy groups of $F(X)$, denoted $H^n(X)$, in the heart of $\C$. For a variety $U$, let $U \to X$ be an open embedding where $X$ is compact (i.e., proper over $k$). If we define the objects $H^*_c(U)$ to be the homotopy groups of
$$F_c(U):=\textup{fib}(F(X) \to F(X \setminus U)),$$ 
then this implies the long exact sequence (\ref{eq:les_compact_support}), and for a compact variety $X$ we have $H^*_c(X)=H^*(X)$. 
Of course, a priori $H^*_c(U)$ depends on the choice of the compactification $X$. We ask under what conditions on the invariant $F:\var \to \C$ this ``cohomology groups with compact support'' $H^*_c(-)$ is well defined. An answer to this question is the condition that $F$ sends abstract blowup squares to homotopy pullback squares in $\C$, as is explained in for example \cite{kelly}; we call this \textit{descent for abstract blowups}. Recall that an abstract blowup square is a pullback square
\begin{center}
\begin{tikzcd}
      E \arrow[d] \arrow[r] & Y \arrow[d, "p"]\\
    C \arrow[r, "i"] & X  
\end{tikzcd}
\end{center}
where $p$ is proper, $i$ a closed immersion and the induced morphism $p:Y \setminus E \to X \setminus C $ is an isomorphism.

We promote this answer to an equivalence of $\infty$-categories of $\C$-valued hypersheaves, where $\C$ is a complete and cocomplete pointed $\infty$-category, using the theory of Grothendieck topologies associated to cd-structures developed by Voevodsky in \cite{voevodsky}. With this theorem we can recover several non-trivial results, using surprisingly little algebraic geometry. One example is the weight complex $W(X)$ of an algebraic variety $X$, which is a bounded complex of pure effective Chow motives, constructed by Gillet and Soulé in \cite{gilletsoule} as an object in the homotopy category of complexes of motives $\textup{Hot}(\mathbf{M})$. A lift of the weight complex from $\textup{Hot}(\mathbf{M})$ to the stable infinity category $K^b(\mathbf{M})$ of complexes of motives up to chain homotopy has been constructed by Bondarko in \cite{bondarko}. With our main theorem, over a field of characteristic zero, it suffices to define the weight complex for smooth and compact varieties, which is trivial. Gillet and Soulé use the weight complex to define a filtration on compactly supported singular cohomology with arbitrary coefficients.
This filtration is trivial for non-singular compact varieties, and coincides with Deligne's weight filtration \cite{hodge_iii} when working over $\mathbb{Q}$. Moreover, Gillet and Soulé define a version of algebraic K-theory with compact support, and show that it has a well-defined filtration which is trivial for non-singular compact varieties. An alternative proof is given by Pascual and Rubió Pons in \cite{p-rp}, using results from Guillén and Navarro-Aznar in \cite{g-na}. Since these constructions are again trivial for smooth and compact varieties, they follow easily from our main theorem as well. Furthermore, the formulation of our result brings to light a mysterious parallel to Bittner's presentation of the Grothendieck ring of varieties \cite{bittner}.

\subsection{Main result}
Suppose $F_c(-)$ is an invariant defined on all algebraic varieties, taking values in a stable $\infty$-category, and contravariantly functorial in morphisms between compact varieties, such that
\begin{equation*}
    F_c(U) = \fib(F_c(X) \to F_c(X\setminus U)) 
\end{equation*}
whenever $U$ is an open subvariety of a compact variety $X$.
It immediately follows that $F_c$ is covariantly functorial in open immersions between arbitrary varieties. Moreover, from Nagata's compactification theorem it follows that $F_c$ is contravariantly functorial in proper morphisms between arbitrary varieties. This motivates the following definition. Let $\span$ be the category with as objects varieties over a field $k$, where a morphism from $X$ to $Y$ is given by a proper map defined on an open subvariety of $X$, or in other words, a span
$$X \xhookleftarrow{i} U \xrightarrow{f} Y $$
where $i$ is an open immersion and $f$ a proper morphism. We denote such a morphism by 
$ X \dashrightarrow Y. $ For $\C$ a pointed $\infty$-category and $F:\span^\op \to \C$ an $\infty$-presheaf, we say that $F$ has the \textit{localisation property} if $F$ maps the sequence
$$X \setminus U \dashrightarrow X \dashrightarrow U$$
to a fibre sequence for any open subvariety $U$ of a variety $X$.

We prove the following theorem, which can be seen as a modernisation and a generalisation of \cite[(2.2.2)]{g-na}. 
\begin{thm}\label{thm:main_intro}
For $\C$ a complete and cocomplete stable $\infty$-category with zero object $*$, consider the following $\infty$-categories:
\begin{itemize}
    \item[(A)] the category of functors $F: \smcomp^\op \to \C$ that satisfy descent for blowup squares and $F(\emptyset)=*$,

    \item[(B)] the category of functors $F: \comp^\op \to \C$ that satisfy descent for abstract blowup squares and $F(\emptyset)=*$,
    \item[(C)] the category of functors $F:\span^\op \to \C$ that satisfy the localisation property.
\end{itemize}
Over a field of characteristic zero, the categories (A) and (B) are equivalent, and over an arbitrary field, the categories (B) and (C) are equivalent.
\end{thm}

\begin{rmk}
  Guillén and Navarro-Aznar show that a functor defined on smooth and compact varieties, compatible with blowups, can be extended uniquely to the category of all varieties and proper maps, if the functor takes values in the homotopy category $Ho\mathcal{D}$ of a \textit{cohomological descent category} $\mathcal{D}$. The  extended functor again takes values in $Ho\D$.  A difference with our result is that we extend functors not to the category of varieties and proper maps, but to the larger category $\span$ which captures not only contravariance in proper maps, but also covariance in open immersions. 
Another advantage of our theorem is that any mention of cohomological descent categories is avoided. A cohomological descent category $\D$ is equipped with simple functors $s:\textup{Fun}(\square^n,\D) \to \D$, which play the role of a ``total complex". We replace $\D$ by a complete and cocomplete pointed $\infty$-category $\C$, and the simple functor by the $\infty$-categorical limit. Moreover, we start with a functor taking values in $\C$ instead of its homotopy category, and the extended functor again takes values in $\C$. For these reasons our result neither implies nor is implied by Theorem (2.2.2) of Guillén and Navarro-Aznar. See also \cite[Remark 3.2.6]{Abdo} about a modernised version of \cite[(2.1.5)]{g-na}, which is proven by Abdó Roig Maranges in his PhD thesis. \end{rmk}

For the equivalence between (A) and (B) in Theorem \ref{thm:main_intro}, in fact $\C$ is only required to be a complete $\infty$-category. The equivalence between (B) and (C) holds if $\C$ is an arbitrary complete and cocomplete pointed $\infty$-category, if in (C) one adds the condition that the functors $\span^\op \to\C$ satisfy descent for abstract blowups; if $\C$ is stable, this is implied by localisation property. Since the proof relies on Nagata's compactification theorem, the equivalence between (B) and (C) holds when instead of varieties, one works with separated finite type Noetherian schemes over a separated finite type Noetherian base scheme.

The categories of functors that occur in Theorem \ref{thm:main_intro} coincide with (sub)categories of hypersheaves. The topologies considered on $\smcomp$ and $\comp$ are variants of the well-studied proper cdh-topology on the category of varieties, as defined in \cite{voevodskycdh}. To our knowledge, the topology that we consider on $\span$ is new, and has some interesting properties; for example, locally every variety is not only smooth, but also compact.

A mystery yet to be explained, is that the categories in Theorem \ref{thm:main_intro} bear a remarkable correspondence to different presentations of the Grothendieck ring of varieties over a field of characteristic zero. This ring, denoted $K_0(\var)$, can be given by the following sets of generators and relations, where in each case the product is given by the product of varieties.
\begin{itemize}
\item[(i)] Generators are smooth and compact varieties, modulo the relations $[\emptyset] =0$ and $[E_C X]+[X] = [Bl_C X] + [C]$ whenever $C$ is a closed subvariety of $X$, $Bl_C X$ is the blowup of $X$ along $C$, and $E_C X$ is the exceptional divisor of the blowup.
    \item[(ii)] Generators are compact varieties, modulo the relations $[\emptyset]=0$ and  $[E] + [X] = [C] + [Y]$ for 
    \begin{center}
        \begin{tikzcd}
           E \arrow[d] \arrow[r]& Y \arrow[d]\\
           C \arrow[r] & X
        \end{tikzcd}
    \end{center}
    an abstract blowup square.
    \item[(iii)] Generators are varieties, modulo the relation $[U] + [X\setminus U] = [X]$ where $U$ is an open subvariety of $X$.
\end{itemize}
The first presentation is Bittner's presentation, the third one is the classical definition. Presentations (ii) and (iii) are equivalent over arbitrary $k$, as we will show in Propostion \ref{propreprvanringofvar}.

\subsection{Outline of the argument}

The categories (A), (B) and (C) in Theorem \ref{thm:main_intro} coincide with (sub)categories of $\C$-valued hypersheaves on $\smcomp,$ $\comp$ and $\span$ respectively, endowed with a suitable Grothendieck topology. In the case of $\smcomp$ and $\comp$, this is the topology associated to a cd-structure, as defined in \cite{voevodsky}, where the cd-structures are given by blowup squares and abstract blowup squares respectively. These cd-structures are complete and regular, implying that sheaves for the associated topologies are exactly the presheaves that send $\emptyset$ to the point and distinguished squares to pullback squares. On $\span$ we can define a cd-structure that consists of abstract blowup squares and \textit{localisation squares}
\begin{center}
    \begin{tikzcd}
       X \setminus U \arrow[d, dashed] \arrow[r,dashed] & X \arrow[d,dashed] \\
       \emptyset \arrow[r,dashed]& U 
    \end{tikzcd}
\end{center}
for $U \subseteq X$ an open subvariety. The category (C) is the category of presheaves on $\span$ that send these distinguished squares to pullback squares, and satisfy $F(\emptyset)=*$. 

However, a complication is that in $\span$ the empty variety $\emptyset$ is a very non-strict initial object. In the topology associated to a cd-structure on any category with an initial object $\emptyset$, by definition the empty sieve covers $\emptyset$. Since in $\span$ there is a morphism $X \dashrightarrow \emptyset$ for any variety $X$, this implies that for such a topology on $\span$ the empty sieve also covers every $X$, so all sheaves for this topology are trivial. To circumvent this, we define the \textit{c-topology} associated to a cd-structure on any category, which does not contain empty covers. We show that if a cd-structure is \textit{c-complete} and \textit{c-regular}, then sheaves for the associated c-topology are exactly the presheaves that send distinguished squares to pullback squares.

If a cd-structure is bounded in addition to being complete and regular, then the hypersheaves for the associated topology are exactly the space-valued presheaves that send $\emptyset$ to a contractible space and distinguished squares to homotopy pullback squares. We define a different criterion on cd-structures which makes this true, namely \textit{being compatible with a dimension function}. This criterion is easy to state and often easy to apply. We also prove the analogous statement for hypersheaves for the c-topology associated to a cd-structure. 

Over a field of characteristic zero, the inclusion $\smcomp\to \comp$ induces an equivalence between categories of sheaves of sets, which implies the equivalence of categories (A) and (B) for any complete $\infty$-category $\C$. However, it can be shown that $\comp \to \span$ is \textit{not} a weakly dense morphism of sites as defined in \cite{denseness}, therefore it does not induce an equivalence of categories of set-valued sheaves. Nonetheless, there is an equivalence of categories of hypersheaves with values in any complete and \textit{pointed} $\infty$-category. Via a general construction, we define a category $\comp_0$ of which $\comp$ is a subcategory, such that in $\comp_0$ there is a morphism $X \to \emptyset$ for every compact variety $X$. We show that the embedding $\comp \to \comp_0$ induces an equivalence of categories of presheaves with values in a cocomplete pointed $\infty$-category $\C$ that send $\emptyset$ to $*$. This implies the equivalence for $\C$-valued hypersheaves that send $\emptyset$ to $*$. Lastly we show that the inclusion $\comp_0 \to \span$ induces an equivalence again at the level of sheaves of sets, implying an equivalence of categories of hypersheaves with values in any complete $\infty$-category. \\

The paper is organised as follows. In Section \ref{section:coarse_topology} and \ref{section:dimension} we modify Voevodsky's theory of topologies associated to cd-structures. In section \ref{section:hypersheaves} we deal with technicalities surrounding hypersheaves with values in a complete $\infty$-category. In Section \ref{section:pointed_category} we give the general construction that results in $\comp_0$ when applied to $\comp$, and prove that the desired equivalence holds between categories of pointed $\infty$-presheaves. In Section \ref{section:sites_with_blowups} we give precise definitions of the cd-structures that we use, and show that they have the desired properties as defined in Section \ref{section:coarse_topology} and Section \ref{section:dimension}. Lastly, in Section \ref{sectmain} we show that inclusions $\smcomp \to \comp$ and $\comp_0 \to \span$ induce equivalences already on the level of sheaves of sets. We summarise the intermediate results and show how they imply Theorem \ref{thm:main_intro} as stated in the introduction. We end with a short proof of the equivalence of the presentations (ii) and (iii) of the Grothendieck ring of varieties. In Section \ref{sectionapp} we discuss some implications of our main theorem.

\subsection{Conventions}
Throughout this paper, a variety is a reduced scheme of finite type over a fixed field $k$. A compact variety is a variety that is proper over $k$. The words category, sheaf and presheaf will always refer to 1-categorical concepts, unless specified otherwise. Simplicial (pre)sheaves are (pre)sheaves of simplicial sets, or equivalently, simplicial objects in the category of set-valued (pre)-sheaves. The word $\infty$-presheaf refers to a presheaf with values in the $\infty$-category of spaces $\S$, or another $\infty$-category. Lastly, hypersheaves are $\infty$-presheaves that satisfy descent for hypercovers. 

\section{Coarse topologies associated to cd-structures} \label{section:coarse_topology}
We recall the following definition from \cite{voevodsky}.
\begin{defn}
For $\X$ a category, a cd-structure on $\X$ is a collection $P$ of distinguished commutative squares
\begin{equation}\label{eq:dist_square}
\begin{tikzcd}
B \arrow[r] \arrow[d]& Y \arrow[d, "p"]\\
A \arrow[r, "i"]& X.
\end{tikzcd}
\end{equation}   
\end{defn}
Given a cd-structure on a category with an initial object, Voevodsky defines $\tau_P$ to be the coarsest topology such that the empty sieve covers the initial object, and such that for each distinguished square as above, the set of morphisms $\{i,p\}$ generates a covering sieve, which is denoted $\langle i,p \rangle$. The disadvantage of this definition is that when the initial object is not strict, there can be many objects with an empty covering. In the extreme case where the initial object is also the terminal object, it follows that every object is covered by the empty cover, and thus all $\tau_P$-sheaves are trivial. We remedy this by replacing $\tau_P$ by the following (strictly coarser) topology. 
\begin{defn}
For $\X$ a category with a cd-structure $P$, the c-topology (short for coarse topology) associated to $P$, denoted $\tau_P^c$, is the coarsest topology such that for any distinguished square of the form (\ref{eq:dist_square}),  the set $\{i,p \}$ generates a covering sieve.
\end{defn}

The following definitions are slight modifications of definitions that can be found in \cite[Section 2]{voevodsky}.

\begin{defn}
Given a cd-structure $P$ on category $\X$, the class of simple $P$-covers $S_P$ is the smallest class of families $\{U_i \to X \}$ in $\X$ such that
\begin{itemize}
    \item[(i)] for $f: X'\rightarrow X$ an isomorphism, $\{f: X'\rightarrow X \} $ is in $S_P$,
    \item[(ii)] for $\{p_i: U_i \rightarrow Y \}$ and $\{q_j: V_j \rightarrow A \}$ in $S_P$ and $Q$ a distinguished square of the form (\ref{eq:dist_square}), the family $\{p\circ p_i:U_i \to X, i\circ q_j: V_j \to X\}$ is in $S_P$.
\end{itemize}
\end{defn}
\begin{defn}
A cd-structure $P$ on an arbitrary category $\X$ is called c-complete if any  $\tau^c_P$-covering sieve contains a simple $P$-cover. 
\end{defn}
For $\X$ a category with a cd-structure $P$ and $X$ an object of $\X$, let let $\rho(X)$ denote the $\tau^c_{P}$-sheafification of the presheaf represented by $X$. 
\begin{defn}
A cd-structure $P$ on an arbitrary category $\X$ is called c-regular if for any distinguished square of the form (\ref{eq:dist_square})
\begin{itemize}
    \item[(a)] the square is cartesian,
    \item[(b)] $i$ is a monomorphism,
    \item[(c)] the morphism of $\tau_P^c$-sheaves 
    \begin{equation}\label{eq:epi_regularity}
    \rho(Y) \amalg \rho(B) \times_{\rho(A)} \rho(B) \to \rho(Y)\times_{\rho(X)}\rho(Y)
    \end{equation}
    is an epimorphism.
\end{itemize}
\end{defn}

Lemma 2.11 in \cite{voevodsky} gives a sufficient criterion for regularity of a cd-structure. From the proof it is clear that the same conditions imply that a cd-structure is c-regular.

\begin{lem}[{\cite[Lemma 2.14]{voevodsky}}]\label{lem:cregularity}
 Let $\X$ be a category with a cd-structure $P$, such that for any distinguished square of the form (\ref{eq:dist_square}) 
\begin{itemize}
    \item[(a)] the square is cartesian,
    \item[(b)] $i$ is a monomorphism,
    \item[(c)] the square 
    \begin{center}
        \begin{tikzcd}
            B \arrow[d] \arrow[r] & Y\arrow[d]\\
            B\times_A B \arrow[r] & Y \times_X Y
        \end{tikzcd}
    \end{center}
    with the vertical arrows diagonals, is distinguished.
\end{itemize}
Then $P$ is a c-regular cd-structure.
\end{lem}

It is equally straightforward to verify that for the other results in \cite[Section 2]{voevodsky}, the analogous statement for the c-topology also holds. We state the relevant results here for future reference.

\begin{lem}[{\cite[Lemma 2.4]{voevodsky}}]\label{lem:c_comp}
A cd-structure $P$ on a category $\mathcal{X}$ is c-complete if and only if for every distinguished square of the form (\ref{eq:dist_square}) and any morphism $f:X'\to X$, the sieve $f^*\langle i,p\rangle$ contains a simple covering.
\end{lem}

\begin{prop}[{\cite[Lemma 2.9 and Proposition 2.15]{voevodsky}}]
Let $P$ be a cd-structure on a category $\mathcal{X}$, and let $\mathcal{C}$ be a category. If $P$ is c-complete and a presheaf $F:\mathcal{X}^\op \to \mathcal{C}$ sends distinguished squares to pullback squares, then $F$ is a $ \tau_P^c$-sheaf. If $P$ is on the other hand c-regular and $G:\mathcal{X}^\textup{op} \to \mathcal{C}$ is a $\tau_P^c$-sheaf, then $G$ sends distinguished squares to pullback-squares.
\end{prop}

\begin{rmk}[{\cite[Corollary 2.16]{voevodsky}}]\label{rmk:pushout_representables}
The second half of the proposition has the following consequence: if $P$ is c-regular, then for a $\tau_P^c$-sheaf $F$ and a distinguished square of the form (\ref{eq:dist_square}), from the identity $F(X) = F(A)\times_{F(B)}F(Y)$ we see that
$$\hom(y_A,F)\times_{\hom(y_B,F)}\hom(y_Y,F) = \hom(y_A \sqcup_{y_B} y_Y,F) = \hom(y_X,F). $$
Applying $\tau^c_{P}$-sheafification, which we denote by $\sharp$, shows that
$$\hom((y_A \sqcup_{y_B}y_Y)^\sharp, F) = \hom({y_A^\sharp}\sqcup_{y_B^\sharp}y_Y^\sharp,F) = \hom(y_X^\sharp,F) $$
which implies that in the category of sheaves, the natural map 
$$ \rho(A)\sqcup_{\rho(B)} \rho(Y) \to \rho(X)$$ is an isomorphism. 
\end{rmk}

Lastly we record the following analogue of \cite[Lemma 2.6, Lemma 2.7, Lemma 2.12 and Lemma 2.13]{voevodsky}.
\begin{lem}
Let $P_1$ and $P_2$ be cd-structures on a category $\X$ that are c-complete (c-regular). Then $P_1 \cup P_2$ is c-complete (c-regular). Let $P$ be a c-complete (c-regular) cd-structure on $\X$, and let $X$ be an object of $X$. Then the induced cd-structure $P/X$ on $\X/X$ is c-complete (c-regular).
\end{lem}

\section{Cd-structures and dimension functions}\label{section:dimension}
In this section we consider a property of cd-structures that is similar to Voevodsky's notion of a cd-structure bounded by a density structure. Like Voevodsky's bounded cd-structures, for a cd-structure with this property, hypersheaves for the associated (c\nobreakdash-)topology are exactly the $\infty$-presheaves that send distinguished squares to homotopy pullback squares, provided that the cd-structure is also (c\nobreakdash-)complete and (c\nobreakdash-)bounded. To our knowledge there is no direct relation between the two concepts, in the sense that given a category with a cd-structure, a reducing density structure does not automatically give rise to a compatible dimension function or vice versa.

With ``$\infty$-presheaves'' on a category $\mathcal{X}$, we mean functors $F:\mathcal{X}^\op \to \S$ where $\S$ is the $\infty$-category of spaces, presented by the classical model structure on $\textup{sSet}$ where the weak equivalences are weak homotopy equivalences and the fibrations are Kan fibrations. The category $\textup{Psh}(\X;\S)$ can then be presented by the global projective model structure on the category of simplicial presheaves  $\textup{Psh}(\X,\textup{sSet})$, with as equivalences the objectwise weak homotopy equivalences, and as fibrations the objectwise Kan fibrations. We denote this model category by $\textup{Psh}(\X,\textup{sSet})_\textup{proj}$.

\begin{defn}
A category with a dimension function is a category $\X$ with an initial object $\emptyset$, and a function 
$$\dim:\textup{Obj}(\X)\to \mathbb{Z}_{\geq-1}$$
such that for $X$ in $\textup{Obj}(\X)$, $\dim(X) = -1$ if and only if $X$ is isomorphic to $\emptyset$. 
\end{defn}
\begin{defn}
For $\X$ a category with a dimension function, a cd-structure $P$ is compatible with the dimension function if there is cd-structure $P' \subseteq P$ such that for every distinguished square of the form (\ref{eq:dist_square}) in $P'$ we have $\dim(A)\leq \dim(X)$, $\dim(Y)\leq \dim(X)$ and $\dim(B)<\dim(X)$, and and such that for every square of the form (\ref{eq:dist_square}) in $P$, the sieve $\langle i,p \rangle$ contains a simple $P'$-cover.
\end{defn}

\begin{defn}
Let $\X$ be a category with a cd-structure $P$. A simplicial presheaf $F$ on $\X$ is called c-excisive with respect to $P$ if for every distinguished square  the square of simplicial sets 
\begin{center}
    \begin{tikzcd}
    F(X) \arrow[r]\arrow[d]&F(Y)\arrow[d]\\
    F(A) \arrow[r] & F(B)
    \end{tikzcd}
\end{center}
is a homotopy pullback square. If $\X$ has an initial object, then the simplicial presheaf $F$ is called excisive with respect to $P$ if it is c-excisive and $F(\emptyset)$ is contractible (this is called a flasque presheaf in \cite{voevodsky} ). 
For $\C$ an $\infty$-category, we call an $\infty$-presheaf $F$ on $\X$ excisive if it sends distinguished squares to pullbacks in $\C$. If $\X$ has an initial object $\emptyset$ and $\C$ a terminal object $*$, then we call $F$ excisive if it is c-excisive and $F(\emptyset) \simeq *$.
\end{defn}
We define two classes of morphisms of simplicial presheaves, following the terminology in \cite{voevodsky}.
\begin{defn}
For $\X$ a category with a Grothendieck topology $\tau$, a morphism of simplicial presheaves $f:F \to G$ on $\X$ is called a $\tau$-local equivalence if \begin{itemize}
    \item[(a)] the morphism of associated sheaves of sets $(\pi_0F)^\sharp \to (\pi_0G)^\sharp $ induced by $f$ is an isomorphism,
    \item[(b)] for any object $X$ of $\X$, 0-simplex $x\in F(X)$ and $n\geq 1$, the induced morphism of associated sheaves of sets 
    $$(\pi_n(F,x))^\sharp \to (\pi_n(G,f(x)))^\sharp $$
    on $\X/X$ is an isomorphism.
 \end{itemize}
\end{defn}
\begin{defn}
For $\X$ a category and $Q$ a commutative square of the form (\ref{eq:dist_square}), the simplicial presheaf $K_Q$ on $\X$ is the homotopy pushout of simplicial presheaves
\begin{center}
    \begin{tikzcd}
       y_B \arrow[r] \arrow[d] & y_A \arrow[d]\\
       y_Y \arrow[r]& K_Q.
    \end{tikzcd}
\end{center}
If $\X$ is endowed with a cd-structure $P$, let $G^0_P$ denote the collection of the canonical morphisms $K_Q \to y_X$ in $\psh(\X;\sset)$ for $Q$ in $P$.
\end{defn}
For $\tau$ a Grothendieck topology on $\X$, we denote the left Bousfield localisation of $\psh(\X;\sset)_{\textup{proj}}$ at the $\tau$-local equivalences by $\psh(\X;\sset)_{\textup{loc}}$. The underlying $\infty$-categories of these simplicial categories are $\psh(\X;\S)$ and $\hsh(\X,\S)$ respectively, and we denote the induced functor of $\infty$-categories by
$$L:\psh(\X;\S) \to \hsh(\X,\S).$$

The following proposition, which is an analogue of \cite[Proposition 3.8]{voevodsky}, shows that if $\X$ has the c-topology associated to a c-complete and c-regular cd-strucure compatible with a dimension function on $\X$, then the model category $\psh(\X;\sset)_{\textup{loc}}$ can also be obtained as the left Bousfield localisation of $\psh(\X;\sset)_{\textup{proj}}$ at the class of morphisms $G_P^0$. Moreover, in this case the hypersheaves are exactly the c-excisive $\infty$-presheaves.
\begin{prop}\label{prop:voevodsky3.8}
Let $\X$ be a category.
\begin{enumerate}
    \item[1] For $P$ a cd-structure, a simplicial presheaf on $\X$ is $G_P^0$-local if and only if it is objectwise fibrant and c-excisive.
    \item[2] For $P$ a c-regular cd-structure, a $G_P^0$-local equivalence is a $\tau_P^c$-local equivalence.
    \item[3] If $\X$ has a dimension function and $P$ is a c-regular, c-complete and compatible with the dimension function, then a $\tau_P^c$-local equivalence is a $G_P^0$-local equivalence.
\end{enumerate}
\end{prop}
We will proof the proposition at the end of this section.

\begin{lem}\label{lem:refinement}
Let $P'\subseteq P$ be cd-structures, such that for every square in $P$ of the form (\ref{eq:dist_square}), the sieve $\langle i,p \rangle$ contains a simple $P'$-cover. Then $P$ and $P'$ have the same associated (c\nobreakdash-)topology, and $P$ is (c\nobreakdash-)complete if and only if $P'$ is.
\end{lem}
\begin{proof}
We prove the statement about c-completeness for the associated c-topology, the proof for the statement about completeness for the associated topology is mostly the same.

It is clear that $P$ and $P'$ have the same associated c-topology. If $P'$ is c-complete, then any $\tau_P$-cover, which is in particular a $\tau_{P'}$-cover, contains a simple $P'$-cover. This is also a simple $P$-cover, so $P$ is c-complete.

Now suppose $P$ is c-complete. Let $S\subseteq S_P$ be the class of simple $P$-covers $\mathcal{U} = \{U_i \to X \}$ such that for every $f:Y \to X$, the pullback $f^*\langle \mathcal{U} \rangle$ of the generated sieve $\langle \mathcal{U}\rangle$ contains a simple $P'$-cover. We show that $S= S_P$. It is clear that $S$ contains $\{f:X'\rightarrow X \}$ for $f$ an isomorphism, since the pullback of the maximal sieve is the maximal sieve. Suppose $\mathcal{U}=\{p_i:U_i \rightarrow Y \}$ and $\mathcal{V}=\{q_j: V_j \rightarrow A \}$ are in $S$, and $Q$ is a distinguished $P$-square of the form (\ref{eq:dist_square}). Then by assumption $\langle i,p\rangle$ contains a simple $P'$-cover $\mathcal{W}=\{r_k:W_k \rightarrow X \}$. For every $k$, we have that $r_k:W_k \rightarrow X$ factors through $i$ or $p$. Let us assume that $r_k$ factors as 
$$W_k \xrightarrow{g}Y \xrightarrow{p}X.$$
Then $g^*\langle\mathcal{U}\rangle$ contains a simple $P'$-covering, and composing this covering with $r_k$ gives, together with the rest of $\mathcal{W}$, a simple $P'$-covering of $X$. Doing this for all $k$ gives a simple $P'$-cover which is in the sieve generated by
$$\{U_i \xrightarrow{p_i}Y \xrightarrow{p}X, V_j \xrightarrow{q_j}A \xrightarrow{i}X\},$$
which shows that $S = S_P$. Now let $\mathcal{U}$ be an an arbitrary $\tau_{P'}$-cover. This is also a $\tau_P$-cover, which contains a family in $S= S_P$ since $P$ is c-complete, and therefore $\mathcal{U}$ contains a simple $P'$-cover.     
\end{proof}
The following terminology is due to Voevodsky.
\begin{defn}
Let $\X$ be a category with a cd-structure $P$. A B.G.-functor on $\X$ with respect to $P$ is a family of presheaves of pointed sets $T_q:\X \to \set_*$ for $q\geq 0$ together with, for a any distinguished square $Q$ of the form (\ref{eq:dist_square}), maps of pointed sets $$\delta_Q: T_{q+1}(B) \to T_q(X)$$ such that 
\begin{itemize}
 \item[(a)] the morphisms $\delta_Q$ are natural with respect to morphisms of distinguished squares,
 \item[(b)] for any $g\geq 0$ and distinguished square of the form (\ref{eq:dist_square}), the sequence of pointed sets 
 $$T_{q+1}(B) \xrightarrow{\delta_Q} T_q(X) \to T_q(A)\times T_q(Y) $$
 is exact.
\end{itemize}
\end{defn}
We prove the following version of \cite[Theorem 3.2]{voevodsky} for cd-structures that are compatible with a dimension function, and also the analogous statement for the associated c-topology.
\begin{thm}[{\cite[Theorem 3.2]{voevodsky}}]\label{thm:voevodsky3.2}
Let $\X$ be a category with a dimension function and $P$ complete (c-complete) cd-structure that is compatible with the dimension function. For any B.G.-functor $(T_q,\delta_Q)$ on $\X$ such that the $\tau_P$-sheaves ($\tau_P^c$-sheaves) associated to $T_q$ are trivial and $T_q(\emptyset) = *$ for all $q \geq 0$, we have $T_q = *$ for all $q\geq 0$. 
\end{thm}
\begin{proof}
By Lemma \ref{lem:refinement}, we can assume that for a distinguished square of the form (\ref{eq:dist_square}) in $P$, we have $\dim(A)\leq \dim(X)$, $\dim(Y)\leq \dim(X)$ and $\dim(B)<\dim(X)$.

Let $T_q$ be a B.G.-functor such that the $\tau_P$-sheaves ($\tau_P^c$-sheaves) $T_q^\sharp$ associated to $T_q$ are trivial. Moreover we assume that $T_q(\emptyset) = *$, in other words, $T_q(X) = *$ for $\dim(X) =-1$. We show by induction that $T_q(X) = *$ for all $X$. 

Assume $T_q(X) = *$ for all $q \geq 0$ and $X$ with $\dim(X) \leq d$. Let $S$ be the family of simple coverings $\mathcal{U} = \{U_i \to X\}$ of objects $X$ in $\X$ such that whenever $\dim(X) \leq d+1$ and $a \in T_q(X)$ is trivial on all $U_i$, we have $a = *$. Clearly $S$ contains all isomorphisms. Now suppose $\{p_i:Y_i \to Y\}$ and $\{q_j:A_j \to A \}$ are in $S$, and we have a distinguished square of the form(\ref{eq:dist_square}). We show that the simple cover $\{p \circ p_i, i \circ q_j \}$ is in $S$. If $\dim(X)>d+1$ this is automatic, so we assume $\dim(X) \leq d+1$. Let $a\in T_q(Y)$ such that $a$ is trivial on all $A_i$ and all $Y_i$. Since $\dim(A) \leq d+1$  and $\dim(Y) \leq d+1$, it follows that $a$ is trivial on $A$ and on $Y$. Since $T_q$ is a B.Q.-functor, there is $b \in T_{q+1}(B)$ such that $a = \delta_Q(b)$. But $\dim(B)\leq d$, so by the induction hypothesis $b=*$ and it follows that $a=*$. So we have shown that $S$ is equal the class of simple covers $S_P$.

Now let $X$ with $\dim(X) = d+1$ and $a\in T_q(X)$. Since $T_q^\sharp = *$, there is a cover $\mathcal{U}=\{U \to Y\}$ such that $a$ is trivial on all $U$. Since $P$ is complete (c-complete), $\mathcal{U}$ contains a simple covering, and since $S_P = S$, it follows that $a = *$. This completes the induction. 
\end{proof}

All other results in \cite[Section 3]{voevodsky} that have to do with bounded cd-structures, are derived from this result. Therefore all these statements are also true for cd-structures that are compatible with a dimension function on $\X$. Moreover, these statements hold for the associated c-topology as well. In many cases, this can be seen by just translating both the statement and the proof, where bounded cd-structure is to be replaced by cd-structure compatible with a dimension function, and if necessary, complete replaced by c-complete, flasque replaced by c-excisive, $t_P$ replaced by $\tau_P^c$ and $G_P$ replaced by $G_P^0$. We record the precise formulations of a few of these statements below.

\begin{lem}[{\cite[Lemma 3.4]{voevodsky}}]\label{lem:voevodsky3.4} Let $\X$ be a category with a cd-structure $P$, and let $F$ be a simplicial presheaf on $\X$ such that $F(U)$ is a Kan complex for all $U$ in $\X$. Then $F$ is c-excisive with respect to $P$ if and only if for a distinguished square as in (\ref{eq:dist_square}), the map of simplicial sets 
\begin{equation}\label{eq:lem3.4}
    \map(y_X,F) \to \map(K_Q,F) 
\end{equation}
induced by $P$ is a weak equivalence.
\end{lem}
\begin{proof}
This follows directly from the proof of \cite[Lemma 3.4]{voevodsky}.
\end{proof}

A case where proving the translated statement actually takes an extra step, is the following lemma. The proof here is almost, but not entirely a translation of the proof of {\cite[Lemma 3.5]{voevodsky}}. 

\begin{lem}[{\cite[Lemma 3.5]{voevodsky}}]\label{lem:voevodksy3.5}
Let $\X$ be a category with a dimension function and $P$ a c-complete cd-structure that is compatible with the dimension function. A morphism of c-excisive presheaves $F \to G$ is a $\tau_P^c$-local equivalence if and only if it is objectwise an equivalence of simplicial sets.
\end{lem}
\begin{proof}
The if-direction is clear. On the other hand, let $f: F\to G$ be a $\tau_P^c$-local equivalence between c-excisive presheaves. Using the global projective model structure on the category of simplicial presheaves on $\X$, we can find a commuting diagram of simplicial presheaves
\begin{center}
    \begin{tikzcd}
            F \arrow[r, "f"] \arrow[d] & G \arrow[d]\\
            F'\arrow[r, "{f'}"]& G
    \end{tikzcd}
\end{center}
such that the vertical maps are objectwise equivalences of simplicial sets, and $f'$ is objectwise a Kan fibration of Kan complexes. Now it suffices to show that $f'$ is objectwise an equivalence of simplicial sets, and we do this by showing that for every object $X$ and 0-simplex $y \in G'(X)$, the fibre $\fib_y(F'(X) \to G'(X))$ is contractible. For such $X$ and $y$, let $H$ be the simplicial presheaf on $\X/X$ given by
$$(U \xrightarrow{u} X) \mapsto \fib_{G'(u)(y)}(F'(U) \to G'(U)).$$
Then $H$ is objectwise a Kan complex, and c-excisive with respect to the cd-structure $P/X$ on $\X/X$ induced by $P$, since the fibre is in this case the same as the homotopy fibre, and therefore commutes with homotopy pullbacks. This cd-structure is complete and compatible with the induced dimension function on $\X/X$. We need to show that $H$ is objectwise contractible. First assume $H(\id_X) \neq \emptyset$ and let $a$ be a 0-simplex in $H(\id_X)$, i.e., a 0-simplex $a$ of $F(X)$ that is mapped to $y$ by $f$. Consider the family of $\set_*$-valued presheaves $\{T_q\}_{q\geq 0}$ on $\X/X$ with $T_q$ given by
$$ (U \xrightarrow{u} X) \mapsto \pi_q(H(U \xrightarrow{u}X), a_{U \to X}). $$
Together with the boundary maps in the long exact sequence associated to a pullback of simplicial sets, these form a B.G.-functor. Since homotopy groups commute with fibres, $T_q(U\to X)$ is the fibre of 
$$\pi_q(F'(U), a|_U) \to \pi_q(G'(U),y|_U) $$
so as a $\set_*$-valued presheaf, $T_q$ is the fibre of the morphism of $\set_*$-valued presheaves 
$$ \pi_q(F(-)' ,a|_{(-)}) \to \pi_q(G(-)',y|_{(-)}).$$
Since $f$ is a $\tau_P^c$-local equivalence, this morphism of presheaves induced by $f$ becomes an isomorphism after $\tau^c_P$-sheafification. Sheafification commutes with finite limits (see e.g. \cite[Page 227]{maclanemoerdijk}), so $T_q^\sharp$ is the fibre of this isomorphism, therefore trivial. To apply Theorem \ref{thm:voevodsky3.2}, unlike Voevodsky we still need to show that $H(\emptyset \to X)$ is contractible, since this is not part of our definition of a c-excisive presheaf. However, since $\emptyset$ is the only object of dimension -1, and $P$ is compatible with the dimension function, the only simple $P$-covering of $\emptyset$ is the identity $\emptyset \to \emptyset$. Since $P$ is c-complete, it follows that the only cover of $\emptyset$ is the trivial one. Therefore for any presheaf $F$ we have $F(\emptyset) = F^\sharp(\emptyset)$. In particular $T_q(\emptyset) = T_q^\sharp(\emptyset) = *$. Now we can conclude that $T_q=*$ for all $q$, showing that $H$ is indeed objectwise contractible.

It remains to show that $H(\id_X)$ is non-empty. Since the sheaf associated to $T_0$ is $*$, and $P/X$ is c-complete, there is a simple cover $\{u_i \to \id_X \}$ of $\id_X$ such that $H(u_i)$ is non-empty for all $i$. We consider the class $S'$ of simple coverings $\{u_i \to u \}$ such that whenever $H(u_i) \neq \emptyset$ for all $i$, we have $H(u) \neq \emptyset$, and we show that $S'$ coincides with the class of all simple coverings, by the same argument with induction over simple covers as is done in the proof of \cite[Lemma 3.5]{voevodsky}. 
\end{proof}

\begin{lem}[{\cite[Lemma 3.7]{voevodsky}}]\label{lem:voevodksy3.7}
Let $P$ be a c-regular cd-structure on a category $\X$. Then all elements of $G_P^0$ are $\tau_P^c$-local equivalences.
\end{lem}
\begin{proof}
For $Q$ a distinguished square (\ref{eq:dist_square}), the square of associated $\tau_P^c$-sheaves is a pushout by Remark \ref{rmk:pushout_representables}. By \cite[Lemma 3.6]{voevodsky} it follows that $p_Q:K_Q \to y_X$ is a $\tau_P^c$-local equivalence.
\end{proof}

We can now give the proof of the main result in this section.
\begin{proof}[Proof of Proposition \ref{prop:voevodsky3.8}]
\begin{enumerate}
    \item[1] Let $F$ be a simplicial presheaf on $\X$. If $F$ is $G_P^0$-local then by definition it is objectwise fibrant, and for every $K_Q \to y_X$ the morphism of simplicial sets (\ref{eq:lem3.4}) is an equivalence, so by Lemma \ref{lem:voevodsky3.4} $F$ is c-excisive. If $F$ is on the other hand objectwise fibrant and c-excisive, then by Lemma \ref{lem:voevodsky3.4} it follows that $F$ is $G_P^0$-local.
    \item[2] For $f:F \to G$ a $G_P^0$-local equivalence, by standard factorisation techniques we can form a diagram
    \begin{equation}\label{diagramin3.8}
        \begin{tikzcd}
            F \arrow[r, "f"] \arrow[d, "e_F"] & G \arrow[d, "e_G"]\\
            \textup{Ex}(f) \arrow[r, "\textup{Ex}(f)"] & \textup{Ex}(G)
        \end{tikzcd}
    \end{equation}
    where $\textup{Ex}(F)$ and $\textup{Ex}(G)$ are $G_P^0$-local, and where $e_F$ and $e_G$ are both $G_P^0$-local equivalences  and, by Lemma \ref{lem:voevodksy3.7} if $P$ is c-regular, $\tau_P^c$-local equivalences (see also \cite[Proof of Proposition 3.8(2)]{voevodsky}). It follows that $\textup{Ex}(f)$ is a $G_P^0$-local equivalence between $G_P^0$-local objects, hence an objectwise weak equivalence, hence a $\tau_P^c$-local equivalence. Therefore $f$ is a $\tau_P^c$-local equivalence. 
    \item[3] Now we assume that $\X$ has a dimension function, and in particular an initial object. For $f:F \to G$ a $\tau_P^c$-local equivalence we can again form a diagram as in (\ref{diagramin3.8}), such that $\textup{Ex}(F)$ and $\textup{Ex}(G)$ are $G_P^0$-local, and such that $e_F$ and $e_G$ are both $G_P^0$-local equivalences and, by \ref{lem:voevodksy3.7} if $P$ is c-regular, $\tau_P^c$-local equivalences. It follows that $\textup{Ex}(f)$ is a $\tau_P^c$-local equivalence between objectwise fibrant and c-excisive presheaves, so by Lemma \ref{lem:voevodksy3.5} it is an objectwise equivalence. Therefore it is a $G_P^0$-local equivalence, and it follows that $f$ is a $G_P^0$-local equivalence.
\end{enumerate}
\end{proof}
\begin{cor}\label{cor:hypersheaves_excisive}
For $\X$ a category with a dimension function, and $P$ a c-regular and c-complete cd-structure compatible with the dimension function, the category of $\tau_P^c$-hypersheaves $\hsh(\X,\S)$ is equivalent to the full subcategory of $\psh(\X;\S)$ on the c-excisive $\infty$-presheaves.
\end{cor}

\begin{rmk}\label{rmk:pushout_representable_hypersheaves}
It follows from Corollary \ref{cor:hypersheaves_excisive} that the analogue of Remark \ref{rmk:pushout_representables} holds for hypersheaves if $\X$ is a category with a dimension function, and $P$ a c-regular and c-complete cd-structure compatible with the dimension function. In that case, for $Q$ a distinguished square of the form (\ref{eq:dist_square}), we have that  
\begin{center}
   \begin{tikzcd}
      Ly_B \arrow[r ]\arrow[d]& Ly_Y \arrow[d] \\
      Ly_A \arrow[r] & Ly_X
   \end{tikzcd}
\end{center}
is a pushout in $\hsh(\X;\S)$.
\end{rmk}

\section{Hypersheaves with values in a complete category}\label{section:hypersheaves}

In this section we study the category of $\C$-valued hypersheaves on a category with a Grothendieck topology $\tau$, for $\C$ an arbitrary $\infty$-category. In the case that $\C$ is complete, we give a different characterisation of this category in Proposition \ref{prop:dagv1.1.12}, and another one in Proposition \ref{prop:C_valued_hypersheaves_excisive} if $\tau$ is the (c\nobreakdash-)topology associated to a (c\nobreakdash-)complete and (c\nobreakdash-)regular cd-structure compatible with a dimension function.

\begin{defn}
Let $\X$ be a site with Grothendieck topology $\tau$, $X$ an object of $\X$ and $U$ a simplicial presheaf on $\X$ with a morphism $U \to y_X$ (where $y_X$ is regarded as a discrete simplicial presheaf). Then $U \to y_X$ is called a hypercover if $U$ is degreewise a coproduct of representables $U_n = \amalg_a y_{U_n^a}$, and $U \to y_X$ is a $\tau$-local trivial fibration.
\end{defn}
For such a hypercover, $U$ is the homotopy colimit of the diagram of discrete simplicial presheaves
$$ \dots \coprod_a y_{U^a_1} \rightrightarrows \coprod_a y_{U^a_0}$$
by \cite[Theorem 1.2]{DHI}. It follows that $U$ is Kan complex-valued and $U \to y_X$ is in $\psh(\X;\S)$.
\begin{defn}
Let $\X$ be a site and $\C$ an $\infty$-category. Let $\hsh(\X;\C)$ denote the subcategory of $\psh(\X;\C)$ consisting of those $F:\X^\op \to \C$ such that for a hypercover $U \to y_X$, where every $U_n$ is a coproduct of representables $\amalg_a y_{U^a_n}$, we have that
$$F(X) \to \prod_a F(U^a_0) \rightrightarrows \prod_a F(U_1^a) \dots $$
is a limit diagram in $\C$. 
\end{defn}

\begin{rmk}\label{rmk:model_structure_hypersheaves}
If $\C$ is the category of Kan complexes $\S$, then by \cite[Theorem 1.3]{DHI} this is equal to the category $\hsh(\X;\S)$ we defined in Section \ref{section:dimension} as underlying $\infty$-category of the localisation of global projective model structure on $\psh(\X;\sset)$ at $\tau$-local equivalences. However, $\hsh(\X;\S)$ can also be presented by Jardine's local injective model structure on the category $\sh(\X;\sset)$ of simplicial sheaves (\cite[Theorem 5]{jardine}). The cofibrations in this model structure are objectwise injections of simplicial sets, and the weak equivalences are $\tau$-local equivalences between simplicial sheaves. Note that this model category only depends on the category of set-valued sheaves $\shv(\X;\set)$. Therefore if $i:\X \to \Y$ induces an equivalence $i^*:\shv(\Y;\set) \to \sh(\X;\set)$, it follows that $i$ also induces an equivalence 
$$ \hsh(\Y;\S) \xrightarrow{\sim} \hsh(\X;\S).$$
\end{rmk}

We recall the following result.
\begin{thm}[{\cite[Theorem 5.1.5.6]{htt}}]\label{thm:htt5.1.5.6}
For $\X$ a site and $\C$ a complete category, precomposition with the Yoneda embedding induces an equivalence of $(\infty,1)$-categories
$$y^*:\Fun^c(\psh(\X;\S)^\op,\C) \xrightarrow{\sim} \Fun(\X^\op, \C)$$
where $\Fun^c$ denotes the subcategory of functors that preserve small limits.
\end{thm}
We can derive from it the following proposition, which is the analogue of \cite[Proposition 1.1.12]{dagv} for hypersheaves.
\begin{prop}\label{prop:dagv1.1.12}
For $\X$ a site and $\C$ a complete $\infty$-category, let $y:\X \to \psh(\X;\S)$ denote the Yoneda embedding, and let $L:\psh(\X;\S) \to \hsh(\X;\S)$ denote the left adjoint to the inclusion, i.e, the hypersheafification. Then precomposition with $L \circ y$ induces an equivalence of $\infty$-categories
$$\Fun^c(\hsh(\X;\S)^\op,\C) \xrightarrow{\sim} \hsh(\X;\C).$$
\end{prop}
\begin{proof}
By \cite[Proposition 5.5.4.20]{htt}, $L$ induces a fully faithful embedding $$L^*:\Fun^c(\hsh(\X;\S)^\op,\C) \to \Fun^c(\psh(\X;\S)^\op,\C).$$ The essential image consists of limit preserving functors $f:\psh(\X;\S)^\op \to \C$ such that for $U \to y_X$ a hypercover, the induced $f(y_X) \to f(U)$ is an equivalence. For such $f$, let $F:\X^\op \to \C$ be the $\infty$-presheaf that is obtained by precomposition with $y$, and let $U\to y_X$ be a hypercover, where $U$ is given by a coproduct of representables $\coprod_a y_{U^a_n}$ in simplicial degree $n$. Recall that $U$ is the homotopy colimit of the diagram
$$ \dots \coprod_a y_{U^a_1} \rightrightarrows \coprod_a y_{U^a_0}.$$
Applying $f$ to this diagram gives
$$\prod_a f( y_{U^a_0}) \rightrightarrows \prod_ a f(y_{U^a_1}) \dots$$
of which the homotopy limit is $f(U)$, since $f$ turns colimits of $\infty$-presheaves into limits in $\C$. By assumption, the map $f(y_X) \to f(U)$ is an equivalence. It follows that
$$F(X) \to \prod_a F(U^a_0) \rightrightarrows \prod_a F(U_1^a) \dots $$
is a homotopy limit diagram in $\C$, so $F$ is a hypersheaf.

On the other hand, if $f:\psh(\X;\S)^\op \to \C$ is a limit preserving functor such that the $\infty$-presheaf $F:\X^\op \to \C$ obtained by precomposition with $L\circ y$ is a hypersheaf, then it follows that  
$$F(X) \to \prod_a F(U^a_0) \rightrightarrows \prod_a F(U_1^a) \dots$$
is a limit. Since $f$ preserves limits,  
$$f(U) \to \prod_a f(y_{U^a_0}) \rightrightarrows \prod_ a f(y_{U^a_1}) \dots$$
is a limit diagram as well, showing that $f(y_X) \to f(U)$ must be an equivalence. It follows that the essential image of $\Fun^c(\hsh(\X;\S)^\op,\C)$  under the equivalence of Theorem \ref{thm:htt5.1.5.6} is $\hsh(\X;\S)$.
\end{proof}

\begin{cor}\label{cor:category_hypersheaves}
If $i:\X \to \mathcal{Y}$ is a morphism of sites such that the induced morphism 
$$i^*:\hsh(\mathcal{Y};\S) \to \hsh(\X;\S) $$
is an equivalence of $\infty$-categories, then for a complete $\infty$-category $\C$, $i$ induces an equivalence of $\infty$-categories
$$i^*:\hsh(\mathcal{Y};\C) \to \hsh(\X;\C). $$
In particular, by Remark \ref{rmk:model_structure_hypersheaves}, it suffices if $i$ induces an equivalence of categories $i^*:\shv(\Y;\set) \to \shv(\X;\set)$.
\end{cor}

We end this section with the analogue of Corollary \ref{cor:hypersheaves_excisive} for hypersheaves with values in an arbitrary category.
\begin{prop}\label{prop:C_valued_hypersheaves_excisive}
Let $\X$ be a category with a dimension function, and $P$ a (c\nobreakdash-)complete and (c\nobreakdash-)regular cd-structure that is compatible with the dimension function. Let $\C$ be a complete $\infty$-category. Then the $\infty$-category $\hsh(\X;\C)$ of $\C$-valued $\tau_P$-hypersheaves ($\tau_P^c$-hypersheaves) is equivalent to the subcategory $\psh_\exc(\X;\C) \subseteq \psh(\X;\C)$ consisting of the (c\nobreakdash-)excisive  $\C$-valued presheaves, i.e., presheaves $F:\X^\op \to \C$ that send distinguished squares to pullback squares in $\C$ (and such that $F(\emptyset)$ is the terminal object).
\end{prop}
\begin{proof}
We show this for $\tau_P^c$, the proof for $\tau_P$ is similar.

We recall from the proof of Proposition \ref{prop:dagv1.1.12} that 
$L:\psh(\X;\S) \to \hsh(\X;\S)$ induces an embedding
$$L^*:\Fun^c(\hsh(\X;\S)^\op,\C) \to \Fun^c(\psh(\X;\S)^\op,\C)$$ of which the image is the subcategory of those small limit-preserving functors $f:\psh (\X,\S)^\op \to \C$ that send hypercovers $U \to y_X$ to an equivalence, or equivalently, that send all $\tau$-local equivalences in $\psh(\X;\S)$ to an equivalence. By Proposition \ref{prop:voevodsky3.8}, this is equivalent to sending all morphisms $K_Q \to y_X$ to an equivalence, for $Q$ a distinguished square. We have seen that $\hsh(\X;\C)$ is the image of this subcategory under the equivalence 
$$y^*:\Fun^c(\psh(\X;\S)^\op,\C) \xrightarrow{\sim} \Fun(\X^\op, \C).$$
We show that this image is equal to $\psh_\exc(\X;\C)$.

For $f \in \Fun^c(\psh(\X;\S)^\op,\C)$ in the image of $L^*$, i.e., $f=g\circ L$ for $g$ in $ \Fun^c(\hsh(\X;\S)^\op,\C)$, let $F$ be the $\infty$-presheaf obtained by precomposing $f$ with the Yoneda embedding. Then for $Q$ a distinguished square of the form (\ref{eq:dist_square}), applying $F$ is applying $g$ to \begin{center}
\begin{tikzcd}
   Ly_B \arrow[r]\arrow[d]& Ly_Y \arrow[d] \\
   Ly_A \arrow[r] & Ly_X
\end{tikzcd}
\end{center}
which is a pushout in $\hsh(\X;\S)$ by Remark \ref{rmk:pushout_representable_hypersheaves}. Since $g$ preserves small limits, this is sent to a pullback square, so $F$ is c-excisive.

On the other hand, let $f \in \Fun^c(\psh(\X;\S)^\op,\C)$ be such that precomposing with $y$ gives a c-excisive $\infty$-presheaf. Then for $Q$ a distinguished square of the form (\ref{eq:dist_square}) we have $f(y_X) = f(y_A) \times_{f(y_B)} f(y_Y) $ and since $f$ preserves small limits, the latter is equal to $ f(K_Q)$.
It follows that $f$ sends the morphism $K_Q \to y_X$ to an equivalence for every distinguished square, so $f$ is in the image of $L^*$.
\end{proof}

\section{Making a category pointed}\label{section:pointed_category}
We will use a general construction to turn the category of compact varieties $\comp$ into a pointed category where the empty variety $\emptyset$ is the zero object. In this section we compare the categories of $\infty$-presheaves on this construction and on the original category. Recall that we call an initial object $\emptyset$ in a category $\C$ \textit{strict} if any morphism $f:X \to \emptyset$ in $\C$ is an isomorphism.

\begin{defn}\label{defn:zero_object}
Let $\X$ a category with strict initial object $\emptyset$, for every $X$ in $\X$ we denote the unique morphism $\emptyset\to X$ by $0$. We define $\X_0$ be the category with the same objects as $\X$,
$\hom_{\X_0}(\emptyset,X)= \hom_\X(\emptyset,X)=\{0\}$, and
$$\hom_{\X_0}(X,Y) = \hom_\X(X,Y) \sqcup \{0 \} $$
if $X$ not isomorphic to $\emptyset$. The composition is determined by the composition in $\X$ and the rule that $f \circ 0 = 0 \circ f = 0$ for all $f$.
\end{defn}
Note that $\X_0$ is a pointed category with zero object $\emptyset$. Let $i:\X \to \X_{0}$ denote the natural inclusion of categories.
 
For $\X$ a category with an initial object $\emptyset$ and $\C$ a pointed ($\infty$-)category, we denote by $\psh(\X;\C)_\emptyset $ the subcategory of $\psh(\X;\C)$ on $\infty$-presheaves $F:\X^\op \to \C$ such that $F(\emptyset)$ is a zero object. 
\begin{prop}\label{prop:sheaves_zero_object}
For $\mathcal{X}$ a category with a strict initial object $\emptyset$ and $\mathcal{C}$ a cocomplete pointed $\infty$-category, there is an equivalence of $\infty$-categories $$i^*:\psh(\X_0;\C)_\emptyset \xrightarrow{\sim} \psh(\X;\C)_\emptyset $$
induced by the inclusion of categories $i:\X \to \X_0$.
\end{prop}
\begin{proof}
Since $\C$ is cocomplete, for $F:\X^\op \to \C$ the left Kan extension $\textup{Lan}_i F$ along $i$ exists, and there is an adjunction
\begin{center}
   \begin{tikzcd}
            \psh (\X;\C) \arrow[r, shift left = .75ex, "{\textup{Lan}_i F}"] & \psh(\X_0;\C). \arrow[l, shift left=.75ex, "i^*"] 
                 
        \end{tikzcd}
\end{center}

For $F:\X^\op\to \C$ and $X\in \X_0$, recall that $\textup{Lan}_i F(X)$ can be  computed as the colimit of $\X^\op\times_{\X_0^\op} (\X_{0}^\op)_{X/} \rightarrow \C$; in other words, as $$\underset{X\to Y}{\textup{colim}}\ F(Y)$$ over the diagram of morphisms $X\to Y$ in $\X_0$, where a morphism from $X \to Y$ to $X\to Z$ in the indexing category is a morphism $Z \to Y$ that makes the obvious triangle commute. This indexing category splits as coproduct of two categories. One is the full subcategory on morphisms $X \to Y$ in $\X$, which has a terminal object $\id:X \to X$, and the other is full subcategory category on zero-morphisms $0:X\to Y$, which has as terminal object $X\to \emptyset$. Therefore we have
$$\textup{Lan}_i F(X) = F(X) \amalg F(\emptyset)$$

Now if $F(\emptyset)=0$, then this shows $\textup{Lan}_i F(X)=F(X)$. Therefore we can restrict the adjunction above to an adjunction 
\begin{center}
   \begin{tikzcd}
            \psh (\X;\C)_\emptyset \arrow[r, shift left = .75ex, "{\textup{Lan}_i F}"] & \psh(\X_0;\C)_\emptyset. \arrow[l, shift left=.75ex, "i^*"] 
                 
        \end{tikzcd}
\end{center}
Moreover, for $F$ in $\psh(\X;\C)_\emptyset$ and for $G$ in $\psh (\X_0;\C)_\emptyset$ we see that the the unit $F\to i^*\textup{Lan}_i F$ and co-unit $\textup{Lan}_i i^*G \to G$ are equivalences. Therefore the adjunction is an equivalence.

\end{proof}

\section{Sites with blowups}\label{section:sites_with_blowups}
In this section we define the categories and the cd-structures that will play a role in proving the main theorem, and show that they have the desired properties. Throughout this section, $k$ is an arbitrary fixed field. Let $\var$ denote the category of algebraic varieties over $k$, let $\comp$ denote the subcategory of compact varieties (i.e., proper over $k$) and let $\smcomp$ denote the subcategory of smooth and compact varieties. We define $\comp_0$ to be the category of compact varieties, with a unique added morphism $X \to \emptyset$ for every compact variety $X$, as in Definition \ref{defn:zero_object}. Lastly, we define $\span$ to be the category with as objects algebraic varieties over $k$, where a morphism $X\to Y$ is a span
$$X\hookleftarrow U \xrightarrow{p} Y  $$
where $U$ is an open subvariety of $X$ and $p$ is a proper morphism. We denote such a morphism also by $X \dashrightarrow_U Y$, or simply by $X \dashrightarrow Y$ if $U$ is clear from the context. The composition of two spans $X \dashrightarrow_U Y$ and $Y \dashrightarrow_V Z$ is defined by taking the pullback of $U \to Y$ along $V \hookrightarrow Y$. In the aforementioned notation, the resulting morphism is $X \dashrightarrow_{U\times_YV } Z$.

\begin{lem}\label{lem:pullbacks_span}
If
\begin{center}
    \begin{tikzcd}
        B \arrow[r]\arrow[d]& Y \arrow[d]\\
        A \arrow[r]&X
    \end{tikzcd}
\end{center}
is a pullback square of varieties where all morphisms are proper, then it is also a pullback when considered as a diagram in $\span$.
\end{lem}
\begin{proof}
Let $P$ be an object with spans  $P \dashrightarrow_{U_0} Y$ and $P \dashrightarrow_{U_1} A$ such that the compositions $P \dashrightarrow_{U_0} X$ and $P \dashrightarrow_{U_1} X$ are equal. Then it must be the case that $U_0 = U_1=U$, and the morphisms $U \to A$ and $U\to Y$ induce a unique morphism $U \to B$. Now $P \dashrightarrow_U  B$ is the unique morphism in $\span$ that makes
\begin{center}
\begin{tikzcd}
P \arrow[rdd, "U"', dashed, bend right] \arrow[rrd, "U" ', dashed, bend left] \arrow[rd, "U" ', dashed] &                                       &                     \\
                                                                                                    & B \arrow[r, dashed] \arrow[d, dashed] & Y \arrow[d, dashed] \\
                                                                                                    & A \arrow[r, dashed]                   & X                  
\end{tikzcd}
\end{center}
commute, which shows that $B$ has the desired universal property in $\span$.
\end{proof}
Recall that a proper morphism of varieties is a monomorphism if and only if it is a closed immersion.
\begin{cor}\label{cormonosinspan}
If a proper morphism of varieties $f:Y \to X$ is a monomorphism in $\var$, i.e., a closed immersion, then it is also a monomorphism when considered as a morphism of $\span$. 
\end{cor}

\begin{defn}
A cartesian square of varieties 
\begin{equation}\label{eq:absblowupsquare}
    \begin{tikzcd}
    E \arrow[d]\arrow[r]& Y\arrow[d, "p"]\\
    C \arrow[r, "i"]& X
    \end{tikzcd}
\end{equation}
is called an abstract blowup square if $i$ is a closed immersion, $p$ is a proper morphism, and if $p$ induces an isomorphism $p: p^{-1}(X \setminus C) \to X \setminus C$.
For $U$ an open subvariety of $X$, the square 
\begin{equation}\label{eq:localisationsquare}
\begin{tikzcd}
    X \setminus U \arrow[r,dashed] \arrow[d, dashed] & X \arrow[d,dashed] \\
    \emptyset  \arrow[r,dashed] & U 
\end{tikzcd}
\end{equation}
in $\span$ is called a localisation square.
\end{defn}

Note that we can view abstract blowup squares as squares in $\span$ as well, and abstract blowup squares of compact varieties as squares in $\comp_0$. The table below summarises which cd-structures we consider on which categories, and how they are denoted.\\

\begin{center}
    \begin{tabular}{|c|c|c|}
         \hline
         \textbf{Category}  & \textbf{Cd-structure} & \textbf{Distinguished squares} \\
         \hline 
         $\smcomp$ & $B$ & Blowup squares of smooth and compact varieties \\ \hline
         $\comp$ & $A^\textup{comp}$ & Abstract blowup squares of compact varieties \\
         \hline
         $\comp_0$ & $A^\textup{comp}$ & Abstract blowup squares of compact varieties \\ \hline
         $\var$ & $A$ & Abstract blowup squares \\ \hline
         $\span$ & $A $ & Abstract blowup squares \\ \hline
         $\span$ & $L$ & Localisation squares \\ \hline 
         $\span$ & $A\cup L$ &  Abstract blowup squares and localisation squares\\ \hline
    \end{tabular}
    \captionof{table}{Overview of cd-structures}\label{tabel}
\end{center}
Here a blowup square is a square of the form (\ref{eq:absblowupsquare}) where $Y$ is the blowup of $X$ along $A$. All these cd-structures will be shown to be compatible with a dimension function that is induced by the following dimension function on $\var$.

\begin{defn}\label{defn:dimension_function}
On the category $\var$ we define the usual dimension function, with $\dim(X)$ the largest integer $d$ such that there exists a chain 
$$V_0 \subset V_1 \subset \dots \subset V_d $$
of distinct, non-empty irreducible subvarieties of $X$. By convention we set $\dim(\emptyset) = -1$. This dimension function induces a dimension function on $\smcomp$, $\comp$, $\comp_0$, and $\span$.
\end{defn}
Standard arguments imply the following.
\begin{lem}\label{lem:dimension}
For $X$ a variety with $C$ a closed subvariety such that $X \setminus C$ is dense in $X$, we have $\dim(X \setminus C) = \dim(X)$ and $\dim(C)<\dim(X)$.
\end{lem}

In the following proposition, we summarise the properties of the cd-structures in the table above.
\begin{prop}
The cd-structures $B$ on $\smcomp$, $A^\textup{comp}$ on $\comp$, $A^\textup{comp}$ on $\comp_0$ and $A\cup L$ on $\span$ are compatible with the dimension function on varieties. The cd-structures $B$ on $\smcomp$ and $A^\textup{comp}$ on $\comp$ are regular and complete, and $A^\textup{comp}$ on $\comp_0$ and $A\cup L$ on $\span$ are c-regular and c-complete. 
\end{prop}
This proposition will follow from the lemmas that occupy the rest of this section. 

\begin{rmk}\label{rmk:hypercompleteness}
For the cd-structure on $\comp$ \cite[Corollary 5.10]{voevodsky} and the variant of this result \cite[Theorem 3.2.5]{AHW} applies, showing that $\infty$-sheaves on $\comp$ coincide with $\infty$-presheaves that send distinguished squares to pullbacks. The verification of the proposition above for $\comp$ can therefore be interpreted as showing that the $\infty$-topos $\sh(\comp;\mathcal{S})$ is hypercomplete.
\end{rmk}

\begin{lem}\label{lem:bounded1}
The cd-structure $A$ on $\var$ is compatible with the dimension function as defined in Definition \ref{defn:dimension_function}.
\end{lem}
\begin{proof}
Let $A'$ be the cd-structure consisting of abstract blowup squares of the form (\ref{eq:absblowupsquare}) such that $\dim(E)<\dim(X)$, $\dim(C)\leq \dim(X)$ and $\dim(Y)\leq \dim(X)$. Now consider an arbitrary abstract blowup square $Q$ as in (\ref{eq:absblowupsquare}); we need to find a simple $A'$-cover in $\langle i,p\rangle $. If $X$ is the union of closed subsets $X_1,X_2$, one not contained in the other, and $X_1$ is irreducible, then we observe that there is an abstract blowup square
\begin{center}
\begin{tikzcd}
X_1\cap X_2 \arrow[r] \arrow[d] & X_1 \arrow[d] \\ 
X_2 \arrow[r] & X.
\end{tikzcd}
\end{center}
Since a $X_1$ irreducible, and $X_1 \cap X_2$ a proper subset of $X_1$, we have $\dim(X_1 \cap X_2)<\dim(X_1)$, so the square is in $A'$. This shows that for $X_0,\dots,X_k$ the irreducible components of $X$, $\{X_i \to X \}_{i=0}^k $ is a simple $A'$-cover. Now consider the squares 
\begin{center}
    \begin{tikzcd}
    E_i \arrow[r]\arrow[d]& Y_i\arrow[d] \\
    C_i \arrow[r]& X_i,
    \end{tikzcd}
\end{center}
denoted $Q\times_X X_i$, which are obtained by pulling back $Q$ along $X_i \to X$. Since $X_i$ is irreducible, we have $\dim(C_i)<\dim (X_i)$. Let $\overline{Y_i\setminus E_i}$ be the closure of $Y_i \setminus E_i$ in $Y_i$. For each $i$, consider the pullback square
\begin{equation}\label{eq:square_i}
     \begin{tikzcd}
    \overline{Y_i \setminus E_i}\cap E_i \arrow[r]\arrow[d]& \overline{Y_i\setminus E_i}\arrow[d] \\
    C_i \arrow[r]& X_i,
    \end{tikzcd}
\end{equation}
where $\overline{Y_i \setminus E_i} \setminus (\overline{Y_i \setminus E_i}\cap E_i) = Y_i\setminus E_i$ is dense in $\overline{Y_i \setminus E_i}$ by construction. By Lemma \ref{lem:dimension} it follows that $\dim (\overline{Y_i\setminus E_i})= \dim(Y_i\setminus E_i) = \dim(X_i \setminus C_i) = \dim(X_i)$ and $\dim ( \overline{Y_i \setminus E_i}\cap E_i) < \dim (X_i)$, so the square (\ref{eq:square_i}) is in $A'$.

Now precomposing the cover $\{X_i \to X \}_{i=0}^k $ with the squares (\ref{eq:square_i}) gives simple  $A'$-cover. Drawing the diagram of arrows from (\ref{eq:square_i}) to $Q\times_X X_i$ to $Q$ shows that the the arrows $C_i \to X_i \to X$ and $\overline{Y_i\setminus E_i} \to X_i \to X$ factor through $i$ and $p$ respectively, so we have found a simple $A'$-cover that is in $\langle i,p \rangle$, showing that $A$ is compatible with the dimension function. 
\end{proof}
\begin{rmk}\label{rmk:bounded}
From the proof of the lemma above, it follows that $A$ as cd-structure on $\span$, and the cd-structure $A^\textup{comp}$ on $\comp$ and on $\comp_0$ are compatible with the dimension function. 
\end{rmk}

\begin{lem}
The cd-structure $B$ on $\smcomp$ is compatible with the dimension function on varieties as defined in Definition \ref{defn:dimension_function}.
\end{lem}
\begin{proof}
Let $B'$ be the cd-structure consisting of blowup squares of the form (\ref{eq:absblowupsquare}) such that $\dim(E)<\dim(X)$, $\dim(C)\leq \dim(X)$ and $\dim(Y)\leq \dim(X)$.

We note that every smooth variety is the disjoint union of smooth, irreducible varieties. If $X$ is the disjoint union of $X_1$ and $X_2$, then the square
\begin{center}
    \begin{tikzcd}
       \emptyset \arrow[r] \arrow[d]& X_1 \arrow[d]\\
       X_2 \arrow[r] & X
    \end{tikzcd}
\end{center}
is a blowup. This shows that for $X_0,\dots,X_k$ the irreducible components of $X$, $\{X_i \to X \}_{i=0}^k $ is a simple $B'$-cover. In general pullbacks do not always exist in $\smcomp$. However, for $Q$ a blowup of the form \ref{eq:absblowupsquare} and $X_i$ a connected component of $X$, the pullback of $Q$ along $X_i \to X$ is again a blowup of smooth varieties. Therefore by the same arguments as in the proof of Lemma \ref{lem:bounded1}, $\langle i,p\rangle$ contains a simple $B'$-cover, which shows that $B$ is compatible with the dimension function.
\end{proof}

\begin{lem}\label{lembounded2}
The cd-structure $L$ on $\span$ is compatible with the dimension function.
\end{lem}
\begin{proof}
Let $L'$ be the cd-structure consisting of squares of the form (\ref{eq:localisationsquare}) where $U$ is dense in $X$. Then we have $\dim(U) = \dim(X)$  and $\dim(X \setminus U)<\dim(X)$ by Lemma \ref{lem:dimension}. Now we show that every localisation square has a refinement that is in $L'$. Indeed, for an arbitrary square of the form (\ref{eq:localisationsquare}), let $\bar U$ be the closure of $U$ in $X$. Composing the closed inclusion $\bar U \dashrightarrow X$ with $X \dashrightarrow U$ shows that a sieve containing $X \dashrightarrow U$ also contains $\bar U \dashrightarrow U$. So $\langle \emptyset \dashrightarrow U, X \dashrightarrow U\rangle$ contains the sieve generated by the square
\begin{center}
    \begin{tikzcd}
    \overline{U} \setminus U \arrow[d,dashed] \arrow[r,dashed] & \overline{U} \arrow[d,dashed] \\
    \emptyset\arrow[r,dashed] & U 
\end{tikzcd}
\end{center}
which is in $L'$.
\end{proof}

We will need the following elementary lemma later on.
\begin{lem}\label{pullbacks}
Let 
\begin{center}
    \begin{tikzcd}
    V \arrow[d, "j", hook] \arrow[r, "f "] & U \arrow[d, hook]\\
    Y \arrow[r] & X
    \end{tikzcd}
\end{center}
be a commutative diagram of varieties, where the vertical maps are open immersions, the horizontal maps are proper, and $j(V)$ is dense in $Y$. Then it is a pullback square.
\end{lem}
\begin{proof}
Consider the morphism $\phi: V \to Y \times_X U $ that arises from this diagram. It is proper since composing with the proper map $Y\times_X U \to U$ gives the proper map $f$. It is also an open immersion, since composing with the open immersion $ i:Y\times_X U \to Y$ gives the open immersion $j:V \to Y$. Therefore the image $\phi(V) \subseteq Y\times_X U$ is both open and closed. If we could write $Y \times_X U = \phi(V) \amalg Z$ with $Z$ non-empty, then $i(Z)$ would be an open in $Y$ disjoint with $i(\phi(V)) = j(V)$, which contradicts $j(V)$ being dense in $Y$. Therefore we have $\phi(V)=Y\times_X U$ and $\phi$ is an isomorphism.
\end{proof}

\begin{lem}
The cd-structures $B$ on $\smcomp$ and $A^\textup{comp}$ on $\comp$ are complete, and $A^\textup{comp}$ on $\comp_0$ and $A$ on $\span$ are c-complete.  
\end{lem}
\begin{proof}
From the proof of \cite[Lemma 4.3]{voevodskycdh} it follows that $B$ is complete.

From \cite[Lemma 2.4]{voevodsky} and the fact abstract blowup squares are stable under pullbacks, it follows that the cd-structure $A^\textup{comp}$ on $\comp$ is complete. 

For $A^\textup{comp}$ on $\comp_0$, we use Lemma \ref{lem:c_comp}. Let $Q$ an abstract blowup of the form (\ref{eq:absblowupsquare}). It is easy to check that pullbacks in $\comp$ are also pullbacks in $\comp_0$. Therefore $f:Z \to X$ a morphism of compact varieties, the sieve $f^*\langle i,p\rangle$ in $\comp_0$ contains the sieve generated by the square $f^* Q$, which is $Q$ pulled back along $f$. If $f$ is the zero morphism $0:Z \to X$, then $f^*\langle i,p\rangle$ is the maximal sieve on $Z$, and therefore a simple cover.

For $A$ on $\span$, it suffices to check that for an abstract blowup square $Q$ of the form (\ref{eq:absblowupsquare}), $f^*\langle i,p \rangle$ contains a simple cover if $f$ is either a proper morphism $Z {\dashrightarrow}_Z X$, or an open inclusion $Z \dashrightarrow_X X$. In the first case, by Lemma \ref{lem:pullbacks_span}, $f^*\langle i,p \rangle$ contains the sieve generated by the square $f^*Q$. In the second case, we have $f$ of the form $Z \dashrightarrow_X X$. By Nagata's compactification theorem, we can factor the morphism of varieties $$Y\xrightarrow{p} X \hookrightarrow Z$$ as an open immersion followed by a proper morphism  $$Y \hookrightarrow Y' \xrightarrow{p'} Z. $$ By replacing $Y'$ by the closure of $Y$ in it, we may assume that $Y$ is dense in $Y'$. Let $C' = \overline{C} \cup (Z \setminus X)$. Then $C'$ is closed in $Z$ and $Z \setminus C' \cong X \setminus C$.
By Lemma \ref{pullbacks}, 
\begin{equation} \label{dichtpullbacksquare}
   \begin{tikzcd}
             Y \arrow[r, hook] \arrow[d, "p"] & Y'\arrow[d, "p'"]\\
             X \arrow[r, hook ]& Z
   \end{tikzcd} 
\end{equation}
is a pullback square, so the inverse image of $Z\setminus C'$ under $p'$ is contained in $Y$, and equal to $p^{-1}(X \setminus C) = Y \setminus E$. This shows that 
\begin{equation}\label{nieuwepullback}
   \begin{tikzcd}
            C'\times_Z Y'  \arrow[r, ] \arrow[d, "p"] & Y'\arrow[d, "p'"]\\
             C' \arrow[r]& Z
   \end{tikzcd} 
\end{equation}
is an abstract blowup square, 

The pullback square (\ref{dichtpullbacksquare}) moreover shows that the composition of spans
$$Y' \xrightarrow{p'}Z \hookleftarrow X$$ is equal to $Y' \dashrightarrow_{Y} X $, which shows that $p'$ is in $f^*\langle i,p \rangle$. Lastly we note that
\begin{center}
    \begin{tikzcd}
        C \arrow[r, hook] \arrow[d,  "i"]& C'\arrow[d]\\
        X \arrow[r, hook] &Z
    \end{tikzcd}
\end{center}
is pullback, which shows that $C'\to Z$ is in $f^*\langle i,p \rangle$. So $f^*\langle i,p \rangle$ contains the simple cover associated to the square (\ref{nieuwepullback}). 
\end{proof}

\begin{lem}\label{lemLccomplete}
The cd-structure $L$ on $\span$ is c-complete. 
\end{lem}
\begin{proof}
By Lemma \ref{lem:c_comp}, it suffices to show that for localisation square of the form (\ref{eq:localisationsquare}) and a morphism $h:Y \dashrightarrow_V U$ in $\span$, the sieve $R = h^*\langle \emptyset \dashrightarrow U, X \dashrightarrow U \rangle$ contains a simple $\tau^c_L$-cover. First we note that $R$ contains the morphism $\emptyset \dashrightarrow Y$ since it factors through $\emptyset \dashrightarrow U$. By Nagata's compactification theorem, we can factor the morphism of varieties $V \to U \hookrightarrow X$ as an open immersion followed by a proper morphism $V \hookrightarrow Z \to X$, where we may assume $V$ is dense in $Z$. Now we glue $Y$ and $Z$ along $V$ to obtain $Y \cup_V Z$. Consider the diagram
\begin{center}
\begin{tikzcd}
Y \arrow[d, hook] & V \arrow[r] \arrow[d, hook] \arrow[l, hook] & U \arrow[d, hook] \\
Y\cup_{V}Z    & Z \arrow[r] \arrow[l, hook]                   & X.             
\end{tikzcd}
\end{center}
We note that the square on the right is pullback of varieties by lemma \ref{pullbacks}; it follows that the outer diagram can be seen as a commuting diagram in $\span$ of morphisms starting in $Y \cup_V Z$ and ending in $U$. This shows that the open immersion $Y\cup_{V}Z \dashrightarrow Y$ composed with $h$ factors through $X \dashrightarrow U$. So $R$ contains the simple cover associated to the localisation square
\begin{center}
    \begin{tikzcd}
    (Y \cup_V Z) \setminus Y \arrow[r,dashed] \arrow[d,dashed] & Y \cup_V Z \arrow[d,dashed] \\
    \emptyset  \arrow[r,dashed] & Y 
\end{tikzcd}
\end{center} which shows that $L$ is c-complete.
\end{proof}

\begin{rmk}\label{rmkcompactelocalisation}
Let $L''$ be the cd-structure consisting of squares of the form (\ref{eq:localisationsquare}) with $X$ compact. Then for every localisation square of the form (\ref{eq:localisationsquare}), the sieve
$$R= \langle \emptyset \dashrightarrow U, X \dashrightarrow U \rangle $$
contains a simple $L''$-cover. Indeed, Let $Y \dashrightarrow X$ be a compactification of $X$, then
\begin{center}
    \begin{tikzcd}
       Y \setminus U \arrow[d, dashed] \arrow[r,dashed] & Y \arrow[d, dashed] \\
       \emptyset  \arrow[r,dashed]& U 
    \end{tikzcd}
\end{center}
is a localisation square and $Y \dashrightarrow U$ is in $R$. It follows from Lemma \ref{lemLccomplete} an Lemma \ref{lem:refinement} that $L''$ is a c-complete cd-structure.
\end{rmk}

\begin{lem}\label{lemregularAAc}
The cd-structure $A^\textup{comp}$ on $\comp$ is regular, and $A^\textup{comp}$ on $\comp_0$ and $A$ on $\span$ are c-regular.
\end{lem}
\begin{proof}
The cd-structure $A$ on $\var$ is regular by \cite[Lemma 2.14]{voevodskycdh}. From the proof it follows that $A^\textup{comp}$ as cd-structure on $\comp$ is regular as well. By Lemma \ref{lem:pullbacks_span} pullbacks in $\var$ are also pullbacks in $\span$, so by Lemma \ref{lem:cregularity} it follows that $A$ is c-regular on $\span$. Similarly, since pullbacks in $\comp$ are pullbacks in $\comp_0$, $A^\textup{comp}$ is c-regular on $\comp_0$.\end{proof}

\begin{lem}
The cd-structure $B$ on $\smcomp$ is regular.
\end{lem}
\begin{proof}
This follows from the proof of \cite[Lemma 4.5]{voevodskycdh}.
\end{proof}

\begin{lem}\label{lemcregularAL}
The cd-structure $A \cup L$ on $\span$ is c-regular.
\end{lem}
\begin{proof} 
For an abstract blowup square, by Lemma \ref{lemregularAAc} we know that the morphism of $\tau^c_A$-sheaves (\ref{eq:epi_regularity}) is an epimorphism. Since further sheafifying from $\tau^c_A$-sheaves to $\tau^c_{A \cup L}$-sheaves preserves finite limits and colimits, the corresponding morphism of $\tau^c_{A \cup L}$-sheaves is also an epimorphism. 

For a localisation square of the form (\ref{eq:localisationsquare}), we claim that the morphism 
\begin{equation}\label{epiforregulaitylocalisation}
\phi: y_X \amalg y_{X \setminus U} \times_{y_\emptyset} y_{X \setminus U} \to y_X \times_{y_U}y_X    
\end{equation}
is an epimorphism in the category of $\tau^c_{A \cup L}$-separated presheaves. Since the plus-construction
$(-)^+:\textup{SPSh}(\span;\set) \to \sh(\Span;\set) $
from $\tau^c_{A \cup L}$-separated presheaves to $\tau^c_{A \cup L}$-sheaves commutes with finite limits and colimits, it then follows that the morphism
$$\rho(X) \amalg \rho(X \setminus U) \times_{\rho(\emptyset)} \rho(X \setminus U) \to \rho(X) \times_{\rho(U)}\rho(X) $$
is an epimorphism in of $\tau^c_{A \cup L}$-sheaves. We defer the proof of the claim to Lemma \ref{sublem3cregularAL}.
\end{proof}

For $\X$ a site, we call a morphism $f:F \to G$ in $\psh(\X;\set)$ \textit{locally surjective} if for every $X$ in $\X$, there exists a cover $\{U_i \to X \}$ such that $f_{U_i}:F(U_i) \to G(U_i)$ is a surjection for all $i$. 
\begin{lem}\label{sublem1cregularAL}
Let $\X$ be a site with Grothendieck topology $\tau$ and $f:F \to G$ a locally surjective morphism in $\psh(\X;\set)$ of $\tau$-separated presheaves. Then $f$ is an epimorphism in the category of $\tau$-separated presheaves $\textup{SPSh}(\X;\set)$. 
\end{lem}
\begin{proof}
Let $f:F \to G$ be a locally surjective morphism of separated presheaves on $\X$, and let $g,g':G \to H$ be morphisms of separated presheaves such that $gf=g'f$. For $X$ an object of $\X$ and $x\in G(X)$, we know that there is a cover $\mathcal{U}=\{p_i:U_i \rightarrow X \}$ and elements $x_i \in F(U_i)$ such that $G(p_i)(x)=f(x_i)$. Then $H(p_i)(g(x)) = gG(p_i)(x) = gf(x_i)$, and on the other hand $H(p_i)(g'(x)) = g'G(p_i)(x)= g'f(x_i)$. So $g(x)$ and $g'(x)$ coincide on the cover $\mathcal{U}$, which implies $g(x) =g'(x)$ since $H$ is separated. 
\end{proof}

We define a family of spans $\{Y_i \dashrightarrow_{U_i}  X \}$ to be jointly surjective if the family of proper morphisms of varieties $\{U_i \to X \}$ is jointly surjective.

\begin{lem}\label{sublem2cregularAL}
Simple $\tau^c_{A \cup L}$-covers are jointly surjective.
\end{lem}
\begin{proof}
We show this by induction. It is clear that an isomorphism is a jointly surjective cover. Suppose we have an abstract blowup square (\ref{eq:absblowupsquare}) and jointly surjective covers $\{ X_i \dashrightarrow_{U_i} C\}$ and $\{Z_j \dashrightarrow_{V_j}  Y\} $. Since $\{C \to X, Y \to X\}$ is jointly surjective, it follows that the simple cover
$$\{X_i \dashrightarrow_{U_i} X, Z_j\dashrightarrow_{V_j} X \} $$
is jointly surjective. Now assume we have a localisation square of the form (\ref{eq:localisationsquare}), and jointly surjective covers $\{ X_i \dashrightarrow \emptyset\}$ and $\{Z_j \dashrightarrow_{V_j} X\} $. Since the $V_j \to X$ are jointly surjective onto $X$, it follows that the restrictions to $V_j \times_X U$ are jointly surjective onto $U$, showing that 
$$\{X_i \dashrightarrow_{\emptyset} U , Z_j \dashrightarrow_{V_j\times_X U} U \} $$
is jointly surjective. This completes the induction.
\end{proof}
\begin{cor}\label{corsublem2cregularAL}
Representable presheaves on $\span$ are $\tau^c_{A\cup L}$-separated.
\end{cor}
\begin{proof} 
Let $X$ and $Y$ be in $\span$, and let 
$$Y \hookleftarrow U \xrightarrow{p} X \textup{ and } Y \hookleftarrow U' \xrightarrow{p'}X$$ be spans in in $y_X(Y)$ that agree on a $\tau^c_{A \cup L}$-cover. Since $A \cup L$ is c-complete, we may assume that the sections agree on a simple cover $\{Z_i \hookleftarrow U_i \xrightarrow{p_i} Y\}$, which is jointly surjective by Lemma \ref{sublem2cregularAL}. So for all $i$
we have that 
$$Z_i \hookleftarrow  p_i^{-1}(U) \to X \textup{ and }Z_i \hookleftarrow p_i^{-1}(U') \to X $$
coincide, in particular $p_i^{-1}(U) = p_i^{-1}(U')$. Let $x \in U \subset A$; then there is $i$ and $y \in U_i$ such that $p_i(y) = x$. So $y$ is in $p_i^{-1}(U) = p_i^{-1}(U')$ and therefore $x\in U'$. By symmetry we have $U = U'$. Moreover $p(x) = p(p_i(y)) = p'(p_i(y)) = p'(x)$. This shows that the spans in $y_X(Y)$ are equal.
\end{proof}

\begin{lem}\label{sublem3cregularAL}
For a localisation square of the form (\ref{eq:localisationsquare}), the morphism $\phi$ in (\ref{epiforregulaitylocalisation}) is an epimorphism of separated presheaves.
\end{lem}
\begin{proof}
By Corollary \ref{corsublem2cregularAL}, we know that $\phi$ is in $\textup{SPSh}(\span;\set)$. To show that it is an epimorphism of separated presheaves, by Lemma \ref{sublem1cregularAL} it now suffices to show that $\phi$ is locally surjective. Let $Y$ be in $\span$. As we have seen in the proof of Lemma \ref{lem:bounded1}, if $Y_0,\dots Y_k$ are the irreducible components of $Y$ then $\{Y_i \to Y \}_{i=0}^k $ is a $\tau^c_{A\cup L}$-cover of $Y$. So we assume that $Y$ is irreducible, and show that $\phi$ evaluated on $Y$ is a surjection of sets. 

An element of $(y_X \times_{y_U} y_X)(Y)$ is a pair of spans $$Y \hookleftarrow V \xrightarrow{p} X \textup{ and } Y \hookleftarrow V' \xrightarrow{p'} X$$ such that $p^{-1}(U) = (p')^{-1}(U)$, and the restrictions of $p$ and $p'$ to this set coincide. One possibility is that $p^{-1}(U)$ and $(p')^{-1}(U)$ are empty; in that case $(p,p')$ is the image of an element of $y_{X \setminus U} \times_{Y_\emptyset}y_{X \setminus U} (Y) = y_{X \setminus U}(Y) \times y_{X \setminus U}(Y)$. If on the other hand $p^{-1}(U) = (p')^{-1}(U)=:V''$ is non-empty, then it is a dense subset of $V \cap V'$.
Let $V'''$ be the largest closed subset of $V \cap V'$ where $p$ and $p'$ coincide; then it contains $V''$ and therefore its closure, which is $V \cap V'$. So $p$ and $p'$ coincide on $V\cap V'$ and we can glue them to one morphism $p'':V \cup V' \to Y$. Composing $p''$ with $V \to V \cup V'$ gives the proper morphism $p:V \rightarrow Y$, so it follows that $V \to V \cup V'$ is a proper morphism. It is also an open immersion whose image is dense (since $V \cup V'$ is open in $Y$ and therefore irreducible), so it must be the identity. It follows that $V = V'$ and $p = p'$, so $(p,p')$ is in the image of $y_X(Y)$ under the diagonal. This shows that $\phi$ is locally surjective.
\end{proof}

\section{Main result}\label{sectmain}
There are inclusion of categories
\begin{equation}\label{eqmorphs}
  \smcomp \to \comp \to \comp_0 \to \span   
\end{equation}
where the morphism $\comp_0 \to \span$ is given by sending a zero morphism $0:X \to Y$ to the span $X \dashrightarrow \emptyset \to Y$. On these categories we consider the cd-structures $B$, $ A^\textup{comp}$, $A^\textup{comp}$ and $A \cup L$ respectively, as given in the table in Figure \ref{tabel}.
In this section we prove the following theorem.
\begin{thm}\label{thmmain}
Let $\mathcal{C}$ be a pointed complete and cocomplete $\infty$-category. Over a field of characteristic zero,  restriction of hypersheaves induces equivalences of $\infty$-categories
$$\hshv(\span;\C)_{\emptyset} \xrightarrow{\sim} \hsh(\comp_0;\C)_\emptyset \xrightarrow{\sim} \hshv(\comp;\C) \xrightarrow{\sim}\hsh(\smcomp;\C).$$
\end{thm}
By Proposition \ref{prop:C_valued_hypersheaves_excisive}, for all these sites the category of $\C$-valued hypersheaves coincides with the categories of (c\nobreakdash-)excisive $\C$-valued presheaves. Together with Proposition \ref{propsimplerstatement}, this implies Theorem \ref{thm:main_intro}.

\begin{rmk}
In light of Remark \ref{rmk:hypercompleteness}, in Theorem \ref{thmmain} the category $\hsh(\comp;\C)$ can be replaced by $\sh(\comp;\C)$. For the other sites, we currently do not know if they are hypercomplete.
\end{rmk}

\begin{proof}[Proof of Theorem \ref{thmmain}] For the third and the first equivalence, by Corollary \ref{cor:category_hypersheaves} it suffices to show the equivalence for the corresponding categories of set-valued sheaves. These equivalences are proven in Lemma \ref{lemmamainresult1} and Lemma \ref{lemmamainresult3}. The equivalence in the middle is Lemma \ref{lemmamainresult2}.
\end{proof}

\begin{lem}[{\cite[Lemma 4.7]{voevodskycdh}}]\label{lemmamainresult1}
Over a field of characteristic zero, restriction along the inclusion $\smcomp \to \comp$ induces an equivalence of categories
$$\shv(\comp,\set) \xrightarrow{\sim} \shv(\smcomp;\set).$$ 
\end{lem}
\begin{proof}
Any field of characteristic zero admits resolution of singularities in the sense of \cite[Definition 3.4]{friedlander_voevodsky}, and this implies that the inclusion $i:\smcompvar \to \compvar$ satisfies the conditions of the Comparison Lemma in \cite{comparison} (see also the proof of \cite[Lemma 4.7]{voevodskycdh}).  We conclude that $i$ induces an equivalence between categories of sheaves. 
\end{proof}

\begin{rmk}\label{rmkoverabdostelling}
In fact the proof of Lemma 4.7 in \cite{voevodskycdh} implies that over a field of characteristic zero, the inclusion $\textup{SmVar} \to \var$ of the category of smooth varieties into the category of all varieties induces an equivalence 
\begin{equation}\label{eqabdostelling}
\sh(\var;\set) \xrightarrow{\sim} \sh(\textup{SmVar};\set),
\end{equation}
where $\textup{SmVar}$ has the topology associated to the cd-structure of blowup squares in $\textup{SmVar}$, and $\var$ the topology associated to the cd-structure $A$. It follows that for $\C$ a complete $\infty$-category, there is an equivalence between $\infty$-categories $\C$-valued hypersheaves on $\textup{SmVar}$ and on $\var$. This is the modernised version of \cite[(2.1.5)]{g-na} which was also proven by Abdó Roig Maranges in his PhD thesis \cite{Abdo}.
\end{rmk}

\begin{lem}\label{lemmamainresult2}
For $k$ any field, and $\C$ a complete and cocomplete pointed $\infty$-category, the inclusion $i:\comp \to \comp_0$ induces an equivalence of $\infty$-categories 
$$i^*:\hsh(\comp_0;\C)_\emptyset \xrightarrow{\sim} \hshv(\comp;\C). $$
\end{lem}
\begin{proof}
By Proposition \ref{prop:sheaves_zero_object}, there is an equivalence of categories of $\C$-valued presheaves  
$$\psh(\comp_0;\C)_\emptyset \to \psh(\comp;\C)_\emptyset.$$
We observe that this equivalence restricts to an equivalence of categories of  $\C$-valued excisive presheaves with respect to $A^\textup{comp}$, and therefore by Proposition \ref{prop:C_valued_hypersheaves_excisive} an equivalence of categories of $\C$-valued hypersheaves $$\hsh(\comp_0;\C)_\emptyset \simeq \hshv(\comp;\C)_\emptyset =  \hshv(\comp;\C).$$
\end{proof}

\begin{lem}\label{lemmamainresult3}
Over any field, the inclusion $\comp_0 \to \span$ induces an equivalence  
$$ \shv(\span;\set) \xrightarrow{\sim}\shv(\comp_0;\set).$$\end{lem}
\begin{proof}
We verify that the inclusion $i:\comp_0 \to \span$ satisfies the conditions of the Comparison Lemma in \cite{comparison}. It is clear that $i$ is cover preserving. 

To check that $i$ is locally full, we observe that for a span $X \dashrightarrow_U Y$ between compact varieties, $X$ can be written as disjoint union of compact varieties $U \amalg X \setminus U $, since $U$ is proper over $Y$ and therefore compact and hence a closed subset of $X$. The square
\begin{center}
    \begin{tikzcd}
       \emptyset \arrow[r] \arrow[d]& X \setminus U \arrow[d]\\
       U \arrow[r] & X
    \end{tikzcd}
\end{center}
is an abstract blowup square, so $\{U \to X, X \setminus U \to X \}$ forms a cover of $X$ in $\comp_0$. Restricted to this cover, the span $X{\dashrightarrow}_U Y$ in $\span$ is the image of the morphisms $0:X\setminus U \rightarrow Y$ and $p:U \rightarrow Y$ in $\comp_0$.

The inclusion $i:\comp_0 \to \span$ is faithful and therefore in particular locally faithful. Next we check that $i$ is locally surjective on objects. Let $X$ be an arbitrary variety and $Y$ a compactification, then clearly $\{\emptyset \dashrightarrow X, Y \dashrightarrow X \}$ is a cover of $X$ in $\span$ by objects in the image of $i$.

Lastly we check that $i$ is co-continuous. Recall the cd-structure $L''$ from Remark \ref{rmkcompactelocalisation}. The cd-structures $A\cup L$ and $A\cup L''$ generate the same associated c-topology, so here we work with the latter cd-structure. Since $A \cup L''$ is c-complete, any cover in $\span$ contains a simple $A \cup L''$ cover. By induction over simple covers, we show that for such a cover $\{X_i {\dashrightarrow}_{U_i} Y \}$ with $Y$ compact, the set of morphisms $Z \to Y$ in $\comp_0$ whose image factors through one of the $X_i {\dashrightarrow}_{U_i} Y $ contains a cover in $\comp_0$. We denote the set of simple covers for which this holds by $S$. 

It is clear that covers containing an isomorphism are in $S$. Now consider a localisation square as in (\ref{eq:localisationsquare}) that is in $L''$, with $U$ compact; note that the open inclusion $U \to X$ is also a proper morphism. Suppose $\{X_i {\dashrightarrow}_{U_i} X \}$ and $\{ V_j {\dashrightarrow}_{\emptyset} \emptyset\}$ are in $S$. Then there is cover $\{f_k:W_k \rightarrow X\}$ in $\comp_0$ such that every $i(f_k)$ factors through a $X_i {\dashrightarrow}_{U_i} X$. The pullbacks $W_k \times_U U \to U$ form a cover of $U$ in $\comp_0$. We note that in $\span$, $i(W_k \times_U U \to U)$ is equal to the composition 
$$W_k \times_U U \dashrightarrow_{W_k \times_U U  } W_k \dashrightarrow_{W_k \times_U U} U .$$
Since 
$W_k  \dashrightarrow_{W_k \times_U U} U$, which is $i(f_k)$ composed with $X {\dashrightarrow}_U U$, factors through $X_i {\dashrightarrow}_{U_i} X {\dashrightarrow}_U U$, it follows that $W_k \times_U U \dashrightarrow_{W_k \times_U U} U$ does as well. It follows that the simple cover 
$$\{X_i{\dashrightarrow}_{U_i} X \dashrightarrow U,  V_j {\dashrightarrow}_{\emptyset} \emptyset \dashrightarrow U \}$$ 
is in $S$.

Now consider an abstract blowup square of the form (\ref{eq:absblowupsquare}) in $\span$ with $X$ compact. Note that in that case $Y$ and $C$ are compact as well. Suppose $\{X_i={\dashrightarrow}_{U_i} Y \}$ and $\{ Z_j {\dashrightarrow}_{V_j} C\}$ are in $S$. Then there exist covers 
$\{X'_k\to Y \} $ and $\{ Z'_l \to C\}$ in $\comp_0$ such that their images factor through some $X_i{\dashrightarrow}_{U_i} Y $ or some $Z_j {\dashrightarrow}_{V_j} C$ respectively. Composing these covers in $\comp_0$ with $p$ and $i$ respectively gives a cover of $X$ such that the image of every morphism factors through a morphism in 
$$\{X_i {\dashrightarrow}_{U_i} Y \dashrightarrow X,  Z_j {\dashrightarrow}_{V_j} C \dashrightarrow X \},$$ therefore the latter is in $S$. This completes the induction, and thus we have shown that $i$ is co-continuous. It follows that $i$ induces an equivalence of categories of sheaves

$$ \shv(\span;\set) \xrightarrow{\sim}\shv(\comp_0;\set).$$

\end{proof} 

The following proposition implies that we can formulate the main result in a simpler way if $\C$ is stable (see Theorem \ref{thm:main_intro}).

\begin{prop}\label{propsimplerstatement}
For $\C$ a stable $\infty$-category, if $F:\span^\op \to \C$ satisfies the localisation property then $F$ also satisfies descent for abstract blowups.
\end{prop}
\begin{proof}
Let $F$ be an $\infty$-presheaf which has the localisation property. For an abstract blowup square of the form (\ref{eq:absblowupsquare}), we observe that in the square
\begin{center}
    \begin{tikzcd}
       F(X) \arrow[r] \arrow[d] & F(Y)\arrow[d] \\
       F(C) \arrow[r] & F(E)
    \end{tikzcd}
\end{center}
the fibres of the vertical maps are equivalent, since $X \setminus C \cong Y \setminus E$. Since $\C$ is stable, it follows that the square is a pullback. 
\end{proof}

We finish this section by giving the following presentation of the Grothendieck ring of varieties.
\begin{prop}\label{propreprvanringofvar}
The Grothendieck ring of varieties over an arbitrary field $k$ can be given by generators $[X]$ for $X$ a compact variety, subject to the the relations $[\emptyset]=0$ and  $[E] + [X] = [C] + [Y]$ for every abstract blowup square of the form (\ref{eq:absblowupsquare}).
\end{prop}
\begin{proof}
Let $K_0(\var)$ denote the ring generated by varieties $X$, subject to the relation $[X] = [U] + [X \setminus U]$ whenever $U$ is an open subvariety of $X$. Let $K_0(\comp)$ denote the ring generated by compact varieties $X$, subject to the relations $[\emptyset]=0$ and  $[E] + [X] = [C] + [Y]$ for every abstract blowup square of the form (\ref{eq:absblowupsquare}). The inclusion $\comp \to \var$ induces a map $f:K_0(\comp) \to K_0(\var)$. This map is well defined: whenever $U$ is an open subvariety of $X$ and both $U$ and $X$ are compact, it follows that $X$ is a disjoint union $X = U \amalg X \setminus U$. Then the square
\begin{center}
    \begin{tikzcd}
       \emptyset \arrow[r] \arrow[d] & X \setminus U \arrow[d]\\
       U \arrow[r] & X
    \end{tikzcd}
\end{center}
is an abstract blowup square, so $[X] = [U]+[X \setminus U]$ in $K_0(\comp)$. 

We show that there is an inverse $g:K_0(\var) \to K_0(\comp)$. For $U$ a variety, let $X$ be an arbitrary compactification of $U$. We define $$g([U])= [X]-[X \setminus U].$$ To show that this map is well-defined, we need to show that $[X]-[X\setminus U] = [Y]-[Y \setminus U]$ for $Y$ any other compactification of $U$. Given two compactifications $X$ and $Y$ of $U$, by Nagata's compactification theorem we can factor $U \to X \times Y$ as an open immersion followed by a proper morphism $U \to Z \to X \times Y$, where we may assume that $U$ is dense in $Z$. Thus, we have a dense compactification that dominates both $X$ and $Z$. Now we can apply Lemma \ref{pullbacks} to the square
\begin{center}
    \begin{tikzcd}
       U \arrow[d,hook] \arrow[r, "="] & U \arrow[d,hook]\\
       Z \arrow[r]& X
    \end{tikzcd}
\end{center}
and conclude that it is a pullback square. It follows that there is a pullback square 
\begin{center}
    \begin{tikzcd}
       Z \setminus U \arrow[r]\arrow[d]& Z \arrow[d]\\
       X\setminus U \arrow[r] & X
    \end{tikzcd}
\end{center}
which is an abstract blowup, therefore we have $[X]+[Z \setminus U] = [Z]+[X \setminus U]$. Similarly, $[Y] + [Z \setminus U] = [Z]+[Y \setminus U]$, and the relation $[X]-[X\setminus U] = [Y]-[Y \setminus U]$ follows. 
It is clear that $f$ and $g$ are each other's inverse.
\end{proof}
\begin{rmk}
From the argument above it follows that there is yet another presentation of $K_0(\var)$, which is given by compact varieties $X$ as generators, subject to the relation $[X]-[Y] = [X']-[Y']$ whenever $Y \subseteq X$ and $Y'\subseteq X'$, and $X \setminus Y \cong X'\setminus Y'$. 

\end{rmk}

\section{Applications}\label{sectionapp}

Over an arbitrary field $k$, the equivalence
$$\hsh(\span;\C)_{\emptyset} \xrightarrow{\sim} \hsh(\comp;\C)$$ implies that any functor $F:\comp^\op \to \C$ that satisfies descent for abstract blowups has a unique ``compactly supported'' extension $F_c:\span^\op \to \C$ which has the localisation property. Applying this to any cohomology theory that satisfies descent for abstract blowups, we obtain the associated compactly supported cohomology theory. If $k$ is of characteristic zero, then it suffices to have a functor defined on $\smcomp$ which satisfies descent for blowups.

\subsection{Compactly supported homotopy algebraic K-theory} \label{subsectionalgK}

Due to Thomason we know that the algebraic K-theory spectrum, restricted to smooth and compact varieties, satisfies descent for blowups \cite{thomason}. Thus, by applying Theorem \ref{thmmain} we obtain the following result.
\begin{thm}[{\cite[Theorem 7]{gilletsoule}}]
    There is a functor
    $$\mathbf{HK}^c:\span^\op \to \textup{Spectra}$$
    which extends the K-theory spectrum on $\smcomp$, and satisfies the localisation property and descent for abstract blowups.
\end{thm}
This gives an alternative construction of the compactly supported homotopy K-theory spectrum defined in \cite{gilletsoule}. Note that $\mathbf{HK}^c$ does \textit{not} coincide with algebraic K-theory on all of $\comp$, since the latter does not satisfy descent for all abstract blowup squares. However, by Haesemeyer's result that the homotopy invariant K-theory spectrum $\mathbf{HK}$ satisfies descent for abstract blowups \cite[Theorem 3.5]{Haes}, it follows that $\mathbf{HK}^c$ coincides with homotopy K-theory on $\comp$.

\subsection{The Künneth ismorphism, Mayer-Vietoris, and the weight complex of motives}
For $\mathcal{X}$ either one of the categories $\smcomp$, $\comp$ or $\span$, and $\C$ a pointed symmetric monoidal $\infty$-category, we say that an $\infty$-presheaf
$F:\mathcal{X}^\op\to\C  $
has \textit{Künneth isomorphisms} if for all $X$ and $Y$ in $\mathcal{X}$ there is an equivalence $$k_{X,Y}:F(X) \otimes F(Y)\to F(X\times Y).$$
\begin{prop}\label{propkunneth}
    Let $\C$ be a complete and cocomplete pointed symmetric monoidal $\infty$-category, where $\otimes$ preserves fibres in each variable. Then for $F:\span^\op\to \C$ in $\hsh(\span; \C)_\emptyset$, $F$ has Künneth isomorphisms if and only if its restriction to $\comp^\op$ does. 
\end{prop}
\begin{proof}
 Assume that $F$ has Künneth isomorphisms $k_{X,Y}$ for $X,Y$ both compact.  To show that an equivalence $k_{X,Y}:F(X) \otimes F(Y) \to F(X\times Y)$ exists for arbitrary $X$ and $Y$ in $\span$, first assume that $Y$ is compact. Let $\overline{X}$ be a compactification of $X$ with complement $C$. From the fibre sequences
$$F(X) \otimes F(Y) \to F(\bar{X}) \otimes F(Y) \to F(C) \otimes F(Y) $$
and
$$F(X\times Y) \to F(\bar{X}\times Y) \to F(C\times Y)  $$
and the existence of $k_{\overline{X},Y}$ and $k_{C,Y}$, we deduce that a morphism 
$$k_{X,Y}:F(X)\otimes F(Y) \to F(X \times Y) $$ exists, which is an isomorphism if $k_{\overline{X},Y}$ and $k_{C,Y}$ are. Having shown the Künneth formula when one of the factors is compact, we now know that $k_{\overline{X},Y}$ and $k_{C,Y}$ exist for arbitrary $Y$, and therefore $k_{X,Y}$ as well.  
\end{proof}

\begin{rmk}\label{rmkkunneth}
    Note that in the proposition above, the Künneth isomorphisms are not required to be functorial in any way. One could hope for an improvement of Theorem \ref{thmmain} for monoidal functors whose underlying functor is a hypersheaf. Such a theorem exists for extending functors from $\smvar$ to $\var$, see \cite[Proposition 6.6]{CH}. The same arguments show that for $\C$ a symmetric monoidal $\infty$-category, restriction induces an equivalence
\begin{equation}\label{eqkunneth}
    \Fun^\otimes_{\tau_{A^\textup{comp}}}(\comp;\C) \xrightarrow{\simeq} \Fun^\otimes_{\tau_B}(\smcomp;\C)
\end{equation}
    
    of categories of strong monoidal functors whose underlying functor is a hypersheaf.

\end{rmk}

Let $F$ be an invariant of varieties, taking values in any $\infty$-category, which is covariantly functorial in open immersions. We say that $F$ satisfies \textit{Mayer-Vietoris} if for $X$ a variety and $U,V$ open subvarieties such that $X = U \cup V$, the square
\begin{equation}\label{MVsquare}
      \begin{tikzcd}
    F(U\cap V)\arrow[r] \arrow[d] & F(U) \arrow[d]\\
    F(V) \arrow[r]& F(X)
    \end{tikzcd}
\end{equation}

is a pullback square.

\begin{prop}\label{propMV}
Every $F$ in $\hsh(\span;\C)_\emptyset$ satisfies Mayer-Vietoris.
\end{prop}
\begin{proof}
Let $X$ be a variety and $U,V$ open subvarieties such that $X = U \cup V$. 
We can extend the square (\ref{MVsquare}) to a diagram
\begin{center}
    \begin{tikzcd}
    F(U\cap V)\arrow[r] \arrow[d] & F(U) \arrow[d] \arrow[r] & F(X \setminus V) \arrow[d]\\
    F(V) \arrow[r]& F(X)\arrow[r] & F(X\setminus V)
    \end{tikzcd}
\end{center}
where the horizontal rows are (co)fibre sequences because of the localisation property. In a stable $\infty$-category, a square is a pullback square if and only if the cofibres of either of the pairs of parallel arrows are equivalent. This is clearly the case here. 
\end{proof}

A use of Proposition \ref{propMV} and Proposition \ref{propkunneth} is illustrated by the following proposition, which recovers the weight complex $W(X)$ for a variety $X$, as defined in \cite{gilletsoule}. Let $\mathbf{M}$ be the category of pure effective Chow motives. Let $K^b(\mathbf{M})$ be the category of bounded complexes of motives, localised at homotopy equivalences. Let
$$M:\smcomp^\op \to  K^b(\mathbf{M}) $$
be the composition of the obvious functors $\smcomp^\op \to \mathbf{M}$ and $\mathbf{M}\to K^b(\mathbf{M})$, taking a smooth and compact variety $X$ to the complex
$$ \dots \to 0 \to  (X,\id_X) \to 0 \to \cdots.$$
\begin{prop}[{\cite[Theorem 2]{gilletsoule}}]
The functor $M$ extends to a functor 
    $$W:\span^\op \to K^b(\mathbf{M})$$ 
    which satisfies descent for abstract blowups, localisation, and Mayer-Vietoris, and has Künneth isomorphisms. 
\end{prop}
\begin{proof}
By the blowup formula for motives \cite{Manin}, we see that $M$ satisfies descent for blowups. Theorem \ref{thmmain} then gives a hypersheaf $W:\span^\op \to K^b(\mathbf{M})$ which satisfies descent for abstract blowups and localisation, and by Proposition \ref{propMV} also Mayer-Vietoris. It is clear that $M$ has Künneth isomorphisms, so by the equivalence (\ref{eqkunneth}) and Proposition \ref{propkunneth} this holds for $W$ as well.
\end{proof}

\subsection{Weight filtrations}\label{subsectionweights}

Let $\C$ be a stable $\infty$-category equipped with a t-structure. Recall that a filtered object in $\C$, i.e., a functor $F:\mathbb{Z} \to \C$, gives rise to a spectral sequence in $\C^\heartsuit$ (see \cite[Section 1.2.2]{ha} for a detailed account). We call a morphism $\eta$ in $\textup{Fun}(\mathbb{Z},\C)$ an $E_1$-equivalence if the induced morphisms between the $E_1$-pages of the associated spectral sequences is a pointwise equivalence; equivalently (assuming the t-structure is non-degenerate), if $\eta$ induces equivalences on all the graded components, where
$$\textup{Gr}_n(F)  = \textup{cofib}(F(n-1) \to F(n)).$$
 Let $\textup{Fil}(\C)$ be the localisation of $\textup{Fun}(\mathbb{Z},\C)$ at the $E_1$-equivalences, as in \cite[Section 2]{gwilliam_pavlov}. Under the appropriate conditions, the spectral sequence assiociated to a filtered object $F$ converges to the homotopy groups of $\colim_n\ F(n)$ and induces a filtration on these groups. 
Let $\tau: \C \to \textup{Fil}(\C)$ be the functor that associates to an object $X$ in $\C$ the canonical filtration 
$$\dots  \to \tau_{\geq n+1}X \xrightarrow{f^n} \tau_{\geq n}X \xrightarrow{f^{n+1}}\tau_{\geq n-1} X \to \cdots$$
Let $\C_-$ be the full subcategory of $\C$ spanned by objects $X$ such that $\tau_{\leq n} X = 0 $ for $n \gg 0$. 
If $\C$ admits sequential colimits, and if these are compatible with the t-structure (according to \cite[Definition 1.2.2.12]{ha}), then for $X$ in $\C_{-}$ the spectral sequence induced by $\tau X$ converges, and the induced filtration on $\pi_n X$ is pure of weight $n$.

\begin{prop}\label{prop:weights}
  Let $\C$ be a complete and cocomplete stable $\infty$-category equipped with a t-structure, such that $\C$ admits sequential colimits compatible with the t-structure.  Let $$F:\smcomp^\op \to \C$$ in $\hsh(\smcomp;\C)$ be such that $F(X)$ is in $\C_-$ for all $X$ in $\smcomp$, and assume that for any blowup square
\begin{center}
    \begin{tikzcd}
       E_C(X) \arrow[r] \arrow[d] & Bl_C(X) \arrow[d, "p"]\\
       C \arrow[r, "i"] & X
    \end{tikzcd}
\end{center}
the induced exact triangle 
$$F(X) \xrightarrow{(F(i),F(p))} F(C) \oplus F(Bl_C(X)) \to F(E_C(X)) \to F(X)[1] $$
splits. Then there is a a functor
$$ F^\tau_c:\span^\op \to \textup{Fil}(\C) $$
that lifts the functor $F_c:\span^\op \to \C$ obtained by extending $F$ using Theorem \ref{thmmain}. For $X$ in $\smcomp$, the induced filtration on $\pi_nF(X)$ is pure of weight n.
\end{prop}
\begin{proof}
It is harmless to replace $\C$ by $\C_-$. Note that the functor $\tau$ defined above does in general not preserve pullbacks; but it is additive and therefore preserves split exact triangles, so the composition $\tau\circ F$ is in $\hsh(\smcomp;\textup{Fil}(\C))$. Then Theorem \ref{thmmain} produces an extension $F_c^\tau$ of this functor to $\span$ which lifts $F_c$.
\end{proof}
Proposition \ref{prop:weights} applies for cohomology theories with Gysin morphisms, in the sense that for $p:Bl_C(X)\to X$ a smooth blow-up, there is a map $g:F(Bl_C(X))\to F(X)$ such that $g\circ F(p)=\id_{F(X)}$. An example is the singular cochain functor $C^*(-;A)$; the induced filtration on singular cohomology with compact support coincides with the weight filtration defined in \cite{gilletsoule}, which is Deligne's filtration \cite{hodge_iii} for $A = \mathbb{Q}$. Another example is the functor $\mathbf{HK}^c$ from Section \ref{subsectionalgK}, because of the existence of Gysin morphisms for algebraic K-theory, see for example \cite{navarro_gysin}. The induced filtration is the weight filtration on algebraic K-theory with compact support, as defined in \cite{gilletsoule} and also \cite{p-rp}. 

\begin{rmk}
    We recall that in a stable $\infty$-category, an exact triangle 
$$ A \to B \to C \to A[1]$$
splits if and only if the morphism $C \to A[1]$ is trivial. Therefore suppose that $F:\smcomp^ \op \to \mathcal{C}$ is a cohomology theory that factors through a motivic category $\mathbf{M} $ equipped with a Chow weight structure (in the sense of Bondarko, see for example \cite{bondarko})
\begin{center}
\begin{tikzcd}
\smcomp^\op \arrow[rr, "F"] \arrow[rd, "h"'] &                               & \C \\
                                               & \mathbf{M} \arrow[ru, "F_0"'] &   
\end{tikzcd}
\end{center}
in such a way that $F_0$ commutes with suspension. The orthogonality condition for weight structures implies that for $C\hookrightarrow X$ a regular closed immersions of smooth and compact varieties, $\textup{Hom}(h(Bl_C(X),h(X)[1]) = 0$; it follows that $F$ sends blowups to split exact triangles so Proposition \ref{prop:weights} applies. This generalises \cite[Theorem 3.1(5)]{bondarkosurvey}.
\end{rmk}

\subsection{The K-theory spectrum of varieties}
As mentioned in the introduction, the three sites occurring in our main theorem correspond to different presentations of the Grothendieck ring of varieties $K_0(\var)$; we shall denote these by $K_0(\smcomp)$, $K_0(\comp)$ and $K_0(\span)$. We can consider the sites $\smcomp$, $\comp$ and $\span$ as \textit{squares categories}, as will be made precise in forthcoming joint work with Campbell, Merling and Zakharevich \cite{squares}. This allows us to construct K-theory spectra $K(\smcomp)$, $K(\comp)$ and $K(\span)$, whose $\pi_0$-groups can be canonically identified with the different presentations of $K_0(\var)$. The spectrum $K(\span)$ is equivalent to the K-theory spectrum of varieties $K(\var)$ as defined previously by Campbell and Zakharevich (\cite{zakharevich}, \cite{campbell}, \cite{campbell_zakharevich}). Moreover we can show
$$K(\comp)\simeq K(\var),$$
thus giving a presentation of the K-theory spectrum of varieties in terms of only compact varieties. The spectrum $K(\smcomp)$ can be seen as a Bittner-style presentation of $K(\var)$, but as of yet is not known to coincide with $K(\var)$. 

For $\X$ either of the sites $\smcomp$, $\comp$ or $\span$, the Yoneda embedding composed with sheafification and stabilization gives a functor
  $$\Sigma_+^\infty Ly:\mathcal{X} \to \mathrm{HSh}(\mathcal{X};\mathrm{Spectra})$$
  which sends distinguished squares to bicartesian squares, see Remark \ref{rmk:pushout_representable_hypersheaves}. In the upcoming work \cite{squares} we show how these functors induce maps of spectra
\begin{center}
    \begin{tikzcd}
     K(\span) \arrow[rr, "\mu_\span"] & & K(\mathrm{HSh}(\span;\mathrm{Spectra})_\emptyset) \arrow[d, "="] \\
     K(\comp) \arrow[rr, "\mu_\comp"] & &K(\mathrm{HSh}(\comp;\mathrm{Spectra})) \arrow[d, "="] \\
     K(\smcomp) \arrow[rr, "\mu_\smcomp"] & & K(\mathrm{HSh}(\smcomp;\mathrm{Spectra})) 
 \end{tikzcd} 
\end{center}
where the spectra on the right are the K-theory spectra of stable $\infty$-categories; we know all of these to be the same by Theorem \ref{thmmain}. 

It would be interesting to compute this subtle motivic measure, and it seems likely that $K(\hsh(\smcomp;\mathrm{Spectra}))$ is more computable than its counterpart $K(\hsh(\span;\mathrm{Spectra})_\emptyset)$. Studying this invariant might also help in understanding whether Bittner's theorem holds at the level of spectra, i.e., if $K(\smcomp)$ is equivalent to the spectra $K(\comp)$ and $K(\span)$.

\section*{Acknowledgements}
I would like to thank my PhD supervisor Dan Petersen for guiding me towards this area of research, answering many questions and providing useful comments on earlier versions of this paper. I am also grateful to Thomas Blom, Alice Hedenlund, Asaf Horev, David Rydh and Robin Benedikt Stoll for helpful conversations and comments. Lastly, I am indebted to anonymous referees for providing useful comments, and in particular for proposing an easier proof of Proposition \ref{prop:sheaves_zero_object}. 

This paper has benefitted from the support of the Institut Mittag-Leffler (Swedish Research Council grant no. 2016-06596) where I participated in the research programs ``Moduli and Algebraic Cycles'' and ``Higher algebraic structures in algebra, topology and geometry''. 

\sloppy

\bibliographystyle{alpha}
\bibliography{bibliography}

\end{document}